\documentclass[12pt,reqno,a4paper]{amsart}
\usepackage{fullpage}
\usepackage[utf8]{inputenc}
\usepackage[T1]{fontenc}
\usepackage{hyperref} 
\usepackage{enumitem}
\usepackage{amsmath,amssymb,amsthm}
\usepackage[dvipsnames]{xcolor}
\usepackage{todonotes}
\usepackage{mathtools}
\usepackage{pgfplots}
\usepackage{pgfplotstable}

\hypersetup{
	colorlinks,
	citecolor = ForestGreen,
	linkcolor = BrickRed,
	urlcolor  = NavyBlue
}

\pgfplotsset{compat=1.15}

\usepackage{fancyhdr}
\advance\footskip0.4cm
\advance\textheight-0.4cm
\calclayout
\newcommand*\patchAmsMathEnvironmentForLineno[1]{%
	\expandafter\let\csname old#1\expandafter\endcsname\csname #1\endcsname
	\expandafter\let\csname oldend#1\expandafter\endcsname\csname end#1\endcsname
	\renewenvironment{#1}%
	{\linenomath\csname old#1\endcsname}%
	{\csname oldend#1\endcsname\endlinenomath}}%
\newcommand*\patchBothAmsMathEnvironmentsForLineno[1]{%
	\patchAmsMathEnvironmentForLineno{#1}%
	\patchAmsMathEnvironmentForLineno{#1*}}%
\AtBeginDocument{%
	\patchBothAmsMathEnvironmentsForLineno{equation}%
	\patchBothAmsMathEnvironmentsForLineno{align}%
	\patchBothAmsMathEnvironmentsForLineno{flalign}%
	\patchBothAmsMathEnvironmentsForLineno{alignat}%
	\patchBothAmsMathEnvironmentsForLineno{gather}%
	\patchBothAmsMathEnvironmentsForLineno{multline}%
}
\usepackage[mathlines]{lineno}

\makeatletter
\def\@seccntformat#1{\hspace*{4mm}%
	\protect\textup{\protect\@secnumfont
		\ifnum\pdfstrcmp{subsection}{#1}=0 \bfseries\fi
		\csname the#1\endcsname
		\protect\@secnumpunct
	}%
}
\makeatother

\makeatletter
\def\section{\@startsection{section}{1}%
\z@{.7\linespacing\@plus\linespacing}{.5\linespacing}%
{\normalsize\scshape\bfseries\centering}}

\renewcommand{\@secnumfont}{\bfseries}
\makeatother

\newcounter{statement}
\newenvironment{statement}[2][!]{%
	\vskip3mm
	\hrule
	\hrule
	\hrule
	\vskip1mm
	\noindent%
	\refstepcounter{statement}%
	\bf#2~\thestatement%
	\ifthenelse{\equal{#1}{!}}{.\ }{~(#1).\ }%
	\it%
}{%
	\vskip1mm
	\hrule
	\hrule
	\hrule
	\vskip2mm
}

\newcounter{algorithm}
\renewcommand{\thealgorithm}{\Alph{algorithm}}
\newenvironment{algorithm}[1][!]{%
	\vskip3mm
	\hrule
	\hrule
	\hrule
	\vskip1mm
	\noindent%
	\refstepcounter{algorithm}%
	\bf%
	Algorithm~\thealgorithm%
	\ifthenelse{\equal{#1}{!}}{.\ }{~(#1).\ }%
	\it%
}{%
	\vskip1mm
	\hrule
	\hrule
	\hrule
	\vskip2mm
}

\newenvironment{theorem}[1][!]{\begin{statement}[#1]{Theorem}}{\end{statement}}
\newenvironment{lemma}[1][!]{\begin{statement}[#1]{Lemma}}{\end{statement}}

\newenvironment{proposition}[1][!]{\begin{statement}[#1]{Proposition}}{\end{statement}}

\newenvironment{remark}[1][!]{\begin{statement}[#1]{Remark}}{\end{statement}}

\newcommand{\defBlackboardCharacter}[1]{\expandafter\def\csname #1\endcsname{\mathbb{#1}}}
\newcommand{\defCaligraphicCharacter}[1]{\expandafter\def\csname #1#1\endcsname{\mathcal{#1}}}

\defBlackboardCharacter{N}
\defBlackboardCharacter{R}
\defBlackboardCharacter{T}
\defBlackboardCharacter{Z}

\defCaligraphicCharacter{C}
\defCaligraphicCharacter{E}
\defCaligraphicCharacter{H}
\defCaligraphicCharacter{I}
\defCaligraphicCharacter{M}
\defCaligraphicCharacter{Q}
\defCaligraphicCharacter{R}
\defCaligraphicCharacter{T}
\defCaligraphicCharacter{U}
\defCaligraphicCharacter{X}

\newcommand{\opt}{{\rm opt}}
\newcommand{\Cstab}{C_{\rm stab}}
\newcommand{\Crel}{C_{\rm rel}}
\newcommand{\Cdrel}{C_{\rm drel}}
\newcommand{\qred}{q_{\rm red}}
\newcommand{\Cmark}{C_{\rm mark}}

\newcommand{\Clin}{C_{\rm lin}}
\newcommand{\qlin}{q_{\rm lin}}
\newcommand{\Ccls}{C_{\rm cls}}

\newcommand{\Copt}{C_\opt}
\newcommand{\Cproblem}{C_{\A}}

\newcommand{\coarse}{H}
\newcommand{\fine}{h}
\newcommand{\refine}{{\tt refine}}
\renewcommand{\vec}[1]{\boldsymbol{#1}}
\newcommand{\normalvec}{\vec{n}}
\newcommand{\A}{A}
\newcommand{\B}{B}
\newcommand{\I}{I}
\renewcommand{\b}{\vec{b}}
\newcommand{\e}{\vec{e}}
\newcommand{\f}{\vec{f}}
\newcommand{\g}{\vec{g}}
\newcommand{\p}{\vec{p}}
\newcommand{\q}{\vec{q}}
\renewcommand{\v}{\vec{v}}
\newcommand{\w}{\vec{w}}
\renewcommand{\u}{\vec{u}}
\newcommand{\z}{\vec{z}}
\renewcommand{\r}{\vec{r}}
\newcommand{\G}{\vec{G}}
\renewcommand{\d}[1]{\,\mathrm{d}#1}
\let\div\relax
\DeclareMathOperator{\div}{div}

\newcommand{\norm}[2][]{#1\| #2 #1\|}
\newcommand{\set}[3][]{#1\{ #2 \, #1| \, #3 #1\}}
\newcommand{\scalarproduct}[3][]{#1\langle #2 \, , \, #3 #1\rangle}
\newcommand{\enorm}[2][]{#1|\!#1|\!#1| #2 #1|\!#1|\!#1|}
\newcommand{\jump}[1]{[\![#1]\!]}

\newcommand{\eqreff}[2]{\stackrel{\mathclap{\eqref{#1}}}{#2}}

\title{Adaptive FEM for parameter-errors in elliptic linear-quadratic parameter estimation problems}

\author{Roland Becker}
\address{Université de Pau et des Pays de l’Adour, IPRA-LMAP, Avenue de l’Université BP 1155, 64013 PAU Cedex, France}
\email{roland.becker@univ-pau.fr}

\author{Michael Innerberger}
\address{TU Wien, Institute for Analysis and Scientific Computing, Wiedner Hauptstr. 8-10/E101/4, 1040 Vienna, Austria}
\email{michael.innerberger@asc.tuwien.ac.at \quad \rm (corresponding author)}

\author{Dirk Praetorius}
\address{TU Wien, Institute for Analysis and Scientific Computing, Wiedner Hauptstr. 8-10/E101/4, 1040 Vienna, Austria}
\email{dirk.praetorius@asc.tuwien.ac.at}

\thanks{{\bf Acknowledgment.} The authors thankfully acknowledge support by the Austrian Science Fund (FWF) through the doctoral school \emph{Dissipation and dispersion in nonlinear PDEs} (grant W1245) and the SFB \emph{Taming complexity in partial differential systems} (grant SFB F65). MI further acknowledges partial support by the \emph{Institut Fran\c{c}ais d'Autriche} in the form of a travel grant.}

\begin{document}

\begin{abstract}
	We consider an elliptic linear-quadratic parameter estimation problem with a finite number of parameters.
	A novel \textsl{a priori} bound for the parameter error is proved and, based on this bound, an adaptive finite element method driven by an \textsl{a posteriori} error estimator is presented.
	Unlike prior results in the literature, our estimator, which is composed of standard energy error residual estimators for the state equation and suitable co-state problems, reflects the faster convergence of the parameter error compared to the (co)-state variables.
	We show optimal convergence rates of our method; in particular and unlike prior works, we prove that the estimator decreases with a rate that is the sum of the best approximation rates of the state and co-state variables.
	Experiments confirm that our method matches the convergence rate of the parameter error.
\end{abstract}

\maketitle


\section{Introduction}\label{sec:intro}

Many applications from science and engineering employ PDE models with a finite number of physical parameters.
As particular model problem, we look at a bounded Lipschitz domain $\Omega \subset \R^d$ with polygonal boundary $\Gamma := \partial \Omega$ for $d \in \N$ and consider
\begin{equation}
\label{eq:problem-strong}
	-\div(\A \nabla u(\p)) = f(\p)
	\quad \text{in } \Omega,
	\qquad
	u(\p) = 0
	\quad \text{on } \Gamma,
\end{equation}
where $A(x) \in \R_\mathrm{sym}^{d \times d}$ is a symmetric and uniformly positive definite matrix, and $f(\p) \in H^{-1}(\Omega) := H_0^1(\Omega)'$ depends linearly on some finite dimensional parameter $\p \in \QQ \subseteq \R^{n_{\QQ}}$ for some fixed $n_{\QQ} \in \N$.
We suppose that the parameter set $\QQ$ is convex and closed.

The parameter $\p$ for a particular application is typically unknown and usually determined by indirect measurements, e.g., thermal diffusion coefficients are determined by measuring temperatures of a sample at different locations.
Mathematically, these measurements are modeled by a vector-valued measurement operator $\G \colon H_0^1(\Omega) \to \CC := \R^{n_\CC}$ for $n_{\CC} \in \N$.
A parameter $\p^\star \in \QQ$ is then chosen such that the simulated measurements $\G(u(\p^\star))$ correspond to experimentally obtained measurements $\G^\star \in \CC$ in a least squares sense.
In this regard, our parameter estimation problem can be viewed as a special case of an optimal control problem with finite dimensional control.

To approximate efficiently the state $u(\p^\star)$, adaptive finite element methods (AFEM), based on local mesh refinement, are an important tool, since they optimize computational effort for a given accuracy; see, e.g., \cite{doerfler96,mns02,bdd04,stevenson07,ckns08,ffp14,axioms}.
Of particular interest are goal-oriented AFEMs, as they are able to achieve faster convergence if one is interested in functional values; see, e.g., \cite{ms09,bet11,fpz16}.

AFEMs have already been studied for numerical optimal control problems on several occasions.
In \cite{bm11, gy17}, adaptivity is driven by residual based \emph{a posteriori} estimators for the energy norm $\enorm{\cdot}$ of (co-)state variables, i.e.,
\begin{equation}
\label{eq:intro-traditional-estimate}
	\norm{\p^\star - \p_\coarse^\star}_{\QQ}
	+ \enorm{u(\p^\star) - u_\coarse(\p^\star_\coarse)}
	+ \enorm{z(\p^\star) - z_\coarse(\p^\star_\coarse)}
	\lesssim
	\eta_{\coarse}(u_{\ell}(\p^\star_\coarse))
	+ \zeta_{\coarse}(z_{\ell}(\p^\star_\coarse)),
\end{equation}
and optimality is shown with respect to these estimators.
However, while the sum of the proposed estimators is an upper bound of the parameter error, its rate of convergence is typically slower than that of the parameter error.
This issue is addressed in~\cite{lc17,gyz16} by replacing the energy norm $\enorm{\cdot}$ by the $L^2$-norm.
Indeed, this improves the convergence rate to match that of the parameter error, but requires strong regularity assumptions to the co-state problem, thus essentially allowing only for convex domains.
Other works (e.g., \cite{bv04, bv05}) focus on error estimators in the context of the dual-weighted residual method.
While the estimators therein are experimentally confirmed to match the convergence rate of the parameter error, these works lack a rigorous convergence analysis.

In this work, we extend the ideas of~\cite{bv04, bv05} and prove \textsl{a priori} and \textsl{a posteriori} error estimates for the parameter error alone, i.e.,
\begin{equation}
\label{eq:intro-our-estimate}
\begin{split}
	\norm{\p^\star - \p_\coarse^\star}_{\QQ}
	&\lesssim
	\Big[ \sum_{i=0}^{n_\QQ} \enorm{u_i - u_{\coarse,i}}^2 \Big]^{1/2}
	\Big[ \sum_{j=1}^{n_\CC} \norm{z_j - z_{\coarse,j}}^2 \Big]^{1/2}\\
	&\lesssim
	\Big[ \sum_{i=0}^{n_\QQ} \eta_{\coarse,i}(u_{\coarse,i})^2 \Big]^{1/2}
	\Big[ \sum_{j=1}^{n_\CC} \zeta_{\coarse,j}(z_{\coarse,j})^2 \Big]^{1/2},
\end{split}
\end{equation}
where $u_i$ and $z_j$ are linked to the state and co-state problem, respectively, and $\eta_{\coarse,i}$ and $\zeta_{\coarse,j}$ are residual-type estimators for energy errors.
Subsequently, this estimate is used to develop an AFEM driven by the \textsl{a posteriori} upper bound. 
Analogously to the dual-weighted residual estimator in~\cite{bv04,bv05}, we observe that the upper bound decays with the same rate as the parameter error.
Moreover, we prove that our AFEM algorithm yields convergence with optimal algebraic rates under much weaker assumptions than those in~\cite{gyz16,lc17}, bridging the gap between theory and practically observed convergence rates in our special situation.
In particular, we stress that optimal convergence rates in~\eqref{eq:intro-our-estimate} are generically quadratic when compared to the optimal rates in~\eqref{eq:intro-traditional-estimate} due to the product structure of the right-hand side in~\eqref{eq:intro-our-estimate} when compared to the additive structure in~\eqref{eq:intro-traditional-estimate}.

\subsection{Outline}
In Section~\ref{sec:problem}, we give details of the problem formulation and its numerical solution, for which we employ a discretization by the finite element method (FEM).
Error estimation as well as our adaptive algorithm are stated in Section~\ref{sec:results}, which also formulates our main results:
\begin{itemize}
	\item we give a novel \textsl{a priori} estimate for the parameter error (Theorem~\ref{th:parameter-apriori});
	\item our estimator is a reliable upper bound of the parameter error (Theorem~\ref{th:parameter-posteriori});
	\item our adaptive algorithm leads to linear convergence with optimal algebraic rates with respect to this estimator (Theorems~\ref{th:linear-convergence}~and~\ref{th:optimal-rates}).
\end{itemize}
The subsequent Sections~\ref{sec:errorbound} and~\ref{sec:optimalrates} are dedicated to the proofs of our main results.
Finally, we provide some numerical experiments in Section~\ref{sec:numerics}, which underline our theoretical results.

\subsection{Notation}
Throughout the paper, we use $a \lesssim b$ if there exists a constant $C > 0$ which is independent of the mesh-width such that $a \leq C\,b$.
We write $a \simeq b$ if $a \lesssim b$ and $b \lesssim a$.
The Euclidean norms on $\QQ$ and $\CC$ are denoted by $\norm{\cdot}_{\QQ}$ and$\norm{\cdot}_{\CC}$, respectively.
Finally, for a function $J \colon \HH_1 \to \HH_2$ between Hilbert spaces $\HH_1, \HH_2$, we denote the gradient and Hessian of $J$ at $v \in \HH_1$ by $J'[v] \colon \HH_1 \to \HH_2$ and $J''[v]  \colon \HH_1 \times \HH_1 \to \HH_2$, respectively.


\section{Parameter estimation problem}\label{sec:problem}

\subsection{Problem formulation}
We consider the linear elliptic PDE problem~\eqref{eq:problem-strong} in weak formulation:
For $\p \in \QQ$, find $u(\p) \in \XX := H_0^1(\Omega)$ such that
\begin{equation}
\label{eq:problem}
	a(u(\p),v)
	:=
	\int_\Omega	\A \nabla u(\p) \cdot \nabla v \d{x}
	=
	F_0(v) + b(\p, v)
	\quad
	\text{for all } v \in \XX.
\end{equation}
We suppose that $\A$ is a piecewise constant and positive definite matrix and that there exist $f_i \in L^2(\Omega)$ as well as $\f_i \in [L^2(\Omega)]^d$ for $i = 0, \ldots, n_\QQ$ such that
\begin{equation*}
	F_0(v) = \int_\Omega f_0 v - \f_0 \cdot \nabla v \d{x}
	\quad \text{and} \quad
	b(\p,v) = \sum_{i=1}^{n_\QQ} p_i \int_\Omega f_i v - \f_i \cdot \nabla v \d{x}
	\quad
	\text{for all } v \in \XX.
\end{equation*}
In particular, this implies that $b(\cdot, \cdot) \colon \QQ \times \XX \to \R$ is linear in both arguments and that, for all $\p \in \QQ$, $F_0, b(\p, \cdot) \in \XX' = H^{-1}(\Omega)$ are linear and continuous functionals on $\XX$.
Under these assumptions, the Lax--Milgram theory implies that problem~\eqref{eq:problem} has a unique solution $u(\p) \in \XX$, which is called \emph{state (variable)}, for all parameters $\p \in \QQ$.

\begin{remark}
	We note that our analysis below readily extends to more general problems, where
	\begin{equation*}
		a(u(\p),v)
		=
		\int_\Omega
		\A \nabla u(\p) \cdot \nabla v + \b \cdot \nabla u(\p) v + c \, u(\p) v
		\d{x},
	\end{equation*}
	and to mixed (inhomogeneous) Dirichlet / Neumann boundary data.
	For such problems, the obvious adaptions are to be made for solution theory and error estimation.
\end{remark}

We suppose that, for all $i = 1, \ldots, n_\CC$, the components $G_i \colon \XX \to \R$ of the measurement operator $\G \colon \XX \to \CC$ take the form
\begin{equation*}
	G_i(v) = \int_\Omega g_i v - \g_i \cdot \nabla v \d{x}
	\quad
	\text{for all } v \in \XX,
\end{equation*}
for some given $g_i \in L^2(\Omega)$ and $\g_i \in [L^2(\Omega)]^d$.

We seek a parameter $\p^\star \in \QQ$ such that the modeled measurements $\G(u(\p^\star))$ match the real measurements $\G^\star$ in some sense.
To this end, we define the residual with respect to the parameter $\p \in \QQ$ as
\begin{equation*}
	\r(\p)
	:= \G(u(\p)) - \G^\star
	\in \CC.
\end{equation*}
Allowing for an additional (regularization) constant $\alpha \geq 0$, we obtain the sought parameter as solution of the \emph{parameter problem}
\begin{equation}
\label{eq:ls-functional}
	J(\p)
	:=
	\frac{1}{2} \norm{\r(\p)}^2_{\CC} + \frac{\alpha}{2} \norm{\p}^2_{\QQ}
	\to \min
	\quad
	\text{in } \QQ,
\end{equation}
with the (regularized) least-squares functional $J$.
Note that, due to $\QQ$ being a closed and convex set and $J$ being quadratic (due to linearity of the residual $\r$), there exists a unique minimizer of~\eqref{eq:ls-functional}, which we call $\p^\star$.
In particular, the corresponding state $u(\p^\star) \in \XX$ also exists and is unique.

\begin{remark}	
	While the linear-quadratic parameter estimation problem~\eqref{eq:problem}--\eqref{eq:ls-functional} is interesting on its own account, we note that such problems also appear as linearization of nonlinear parameter estimation problems, e.g., in the course of a Gauss--Newton iteration.
	Nonlinear parameter estimation problems in the context of adaptive iterative linearized algorithms as presented, e.g., in~\cite{hpw2021} will be the subject of future work. 
\end{remark}

\subsection{Solution components}
Due to the linearity of~\eqref{eq:problem} with respect to the parameter $\p$, we can decompose the solution $u(\p)$ into components.
To this end, we denote by $\e_i \in \R^{n_\QQ}$ the $i$-th unit vector and define
\begin{equation}
\label{eq:primal-components}
\begin{split}
	u_0 \in \XX \colon
	\qquad
	a(u_0, v)
	&=
	F_0(v) \hphantom{b(\e_i, v)}
	\text{for all }
	v \in \XX,\\
	u_i \in \XX \colon
	\qquad
	a(u_i, v)
	&=
	b(\e_i, v) \hphantom{F_0(v)}
	\text{for all } v \in \XX
	\text{, and all } i = 1, \ldots, n_\QQ.
\end{split}
\end{equation}

Considering the solution $u(\p)$ to~\eqref{eq:problem} as a mapping $u \colon \QQ \to \XX$, we can compute the derivative $u' \colon \QQ \to L(\QQ, \XX)$ with respect to the parameter to see that, for every $\p,\q \in \QQ$, the function $u'(\p) := u'[\q](\p)$ is independent of the linearization point $\q$ and solves
\begin{equation}
\label{eq:uprime-problem}
	a(u'(\p), v) = b(\p, v)
	\qquad
	\text{for all } v \in \XX.
\end{equation}
Due to linearity of $b(\cdot, \cdot)$ in both arguments and $\p = \sum_{i=1}^{n_{\QQ}} p_i \e_i$, we have that
\begin{equation}
\label{eq:component-combination}
	u(\p)
	= u_0 + \sum_{i=1}^{n_\QQ} p_i u_i,
	\qquad
	u'(\p)
	= \sum_{i=1}^{n_\QQ} p_i u_i.
\end{equation}

\subsection{Least squares system and solution}
By the assumptions on $\QQ$ and quadraticity of $J \colon \QQ \to \R$, the problem~\eqref{eq:ls-functional} is a convex optimization problem.
The first-order necessary condition for the minimizer $\p^\star \in \QQ$ of \eqref{eq:ls-functional} is
\begin{equation}
\label{eq:first-order-condition}
	J'[\p^\star](\p-\p^\star) \geq 0
	\quad
	\text{for all } \p \in \QQ.
\end{equation}
Since the least squares functional is quadratic, the Hessian $J'' \in \R^{n_\QQ \times n_\QQ}$ is constant and, in particular, independent of the evaluation point.
We assume that there exists a constant $\kappa > 0$ such that the second-order sufficient condition 
\begin{equation}
\label{eq:locally-convex}
	J'' (\p,\p)
	\geq
	\kappa \norm{\p}_{\QQ}^2
	\quad
	\text{for all } \p \in \QQ
\end{equation}
holds, i.e., that every solution to~\eqref{eq:first-order-condition} is indeed a minimizer; see the general text~\cite{nocedal-wright} for existence and uniqueness of minimizers and Remark~\ref{rem:ls-solve}{\rm (ii)} below for the role of $\kappa$.

From linearity of the measurement functional $\G \colon \XX \to \CC$, we infer that $\G'[v](u'(\q)) = \G(u'(\q))$ for all $v \in \XX$ and $\q \in \QQ$.
Thus, it holds that
\begin{equation}
\label{eq:Jprime}
\begin{split}
	J'[\p^\star](\q)
	&=
	\scalarproduct[\big]{\r(\p^\star)}{\r'[\p^\star](\q)}_{\CC} + \alpha \scalarproduct{\p^\star}{\q}_{\QQ}\\
	&=
	\scalarproduct[\big]{\G(u(\p^\star)) - \G^\star}{\G(u'(\q))}_{\CC}
	+ \alpha \scalarproduct{\p^\star}{\q}_{\QQ}.
\end{split}
\end{equation}
Defining $\B \in \R^{n_\QQ \times n_\CC}$ by $\B_{ij} := G_j(u_i)$, we can use linearity in the last equation to obtain that
\begin{align}
\nonumber
	J'[\p^\star](\q)
	&=
	\sum_{i,j=1}^{n_\QQ} \sum_{k=1}^{n_\CC}
	\Big[
		\big( G_k(u_0) + p^\star_i G_k(u_i) - G^\star_k \big) \big( q_j G_k(u_j) \big)
	\Big]
	+ \alpha \scalarproduct{\p^\star}{\q}_{\QQ}\\
\label{eq:Jprime-explicit}
	&=
	\q^\intercal (\B \B^\intercal + \alpha \I) \p^\star + \q^\intercal \B (\G(u_0) - \G^\star).
\end{align}
From this representation, a solution to~\eqref{eq:first-order-condition} can be computed, heeding the constraints induced by $\QQ$; see~\cite{nocedal-wright} for a comprehensive treatment of algorithms for such (possibly non-linear) numerical optimization problems.

\begin{remark}
\label{rem:ls-solve}
	{\rm (i)}
	For the unconstrained case $\QQ = \R^{n_{\QQ}}$, solving the optimality condition~\eqref{eq:first-order-condition} simplifies to solving the (linear) least-squares system 
	\begin{equation*}
		(\B \B^\intercal + \alpha \I) \p^\star = \B (\G^\star - \G(u_0)).
	\end{equation*}
	
	{\rm (ii)}
	From the explicit representation
	\begin{equation}
	\label{eq:hessian-explicit}
		J''(\p,\p)
		=
		\p^\intercal \big( \B \B^\intercal + \alpha \I \big) \p,
	\end{equation}
	one infers two sufficient conditions such that \eqref{eq:locally-convex} holds:
	The first one is that the system matrix $\B \B^\intercal$ has full rank, in which case $\alpha = 0$ is an admissible choice.
	In particular, this requires $n_\CC \geq n_\QQ$, i.e., more measurements than parameters.
	The second condition is that $\alpha > 0$.
\end{remark}

\subsection{FEM discretization}
For a conforming triangulation  $\TT_\coarse$ of $\Omega \subset \R^d$ into compact simplices and a polynomial degree $k \ge 1$, let
\begin{equation*}
	\XX_\coarse
	:= 
	\set{v_\coarse \in H^1_0(\Omega)}{\forall T \in \TT_\coarse \colon ~~ v_\coarse|_T \text{ is a polynomial of degree } \le k}.
\end{equation*}
To obtain a conforming finite element approximation $u(\p) \approx u_\coarse(\p) \in \XX_\coarse$ for $\p \in \QQ$, we consider the Galerkin discretization of~\eqref{eq:problem}, which reads:
Find $u_\coarse(\p)$ such that
\begin{equation*}
\label{eq:discrete-problem}
	a(u_\coarse(\p), v_\coarse) = F_0(v_\coarse) + b(\p, v_\coarse)
	\quad
	\text{ for all } v_\coarse \in \XX_\coarse.
\end{equation*}
Moreover, the functions $u_{\coarse,0}, u_{\coarse,i} \in \XX_\coarse$ are defined analogously to~\eqref{eq:primal-components}, which is why~\eqref{eq:component-combination} holds accordingly.

On $\TT_\coarse$, an approximation of the continuous parameter $\p^\star$ can be obtained by minimizing a discretized version of the least-squares functional~\eqref{eq:ls-functional}:
\begin{equation}
\label{eq:discrete-ls-functional}
	J_\coarse(\p)
	:=
	\frac{1}{2} \norm{\r_\coarse(\p)}^2_{\CC} + \frac{\alpha}{2} \norm{\p}^2_{\QQ}
	\quad \text{with} \quad
	\r_\coarse(\p)
	:= \G(u_\coarse(\p)) - \G^\star.
\end{equation}
The minimizer $\p^\star_\coarse \in \QQ$ of the discrete least-squares functional satisfies
\begin{equation}
\label{eq:first-order-condition-discrete}
	J_\coarse'[\p^\star_\coarse](\p-\p^\star_\coarse) \geq 0
	\quad
	\text{for all } \p \in \QQ.
\end{equation}
We note that a discrete representation of~\eqref{eq:discrete-ls-functional} can be derived in complete analogy to~\eqref{eq:Jprime-explicit}, with $\B_{\coarse, ij} := G_j(u_{\coarse,i})$.

\subsection{Co-state components}
In the following adaptive algorithm and its analysis we also need information about the measurement operators.
To this end, we introduce the \emph{co-state components}
\begin{equation}
\label{eq:def-dual-components}
	z_j \in \XX \colon
	\qquad
	a(v, z_j)
	=
	G_j(v)
	\qquad
	\text{for all }
	v \in \XX, j = 1, \ldots, n_\CC
\end{equation}
and their discretizations $z_{\coarse,j} \in \XX_\coarse$.

To give a concise presentation of our analysis, we further define for $\p \in \QQ$ the functions
\begin{equation*}
\label{eq:dual-solution}
	z(\p)
	:=
	\sum_{j=1}^{n_\CC} r_{\coarse,j}(\p) z_j
	\quad \text{and} \quad
	z_\coarse(\p)
	:=
	\sum_{j=1}^{n_\CC} r_{\coarse,j}(\p) z_{\coarse,j}
\end{equation*}
and note that they satisfy
\begin{equation}
\label{eq:z-equation}
\begin{split}
	a(v, z(\p))
	&=
	\scalarproduct[\big]{\r_\coarse(\p)}{\G(v)}_{\CC}
	\quad
	\text{for all } v \in \XX,\\
	a(v_\coarse, z_\coarse(\p))
	&=
	\scalarproduct[\big]{\r_\coarse(\p)}{\G(v_\coarse)}_{\CC}
	\quad
	\text{for all } v_\coarse \in \XX_\coarse.
\end{split}
\end{equation}
Taking the derivative of the last equations with respect to the parameter, we obtain
\begin{equation}
\label{eq:zprime-equation}
\begin{split}
	a(v, z'(\p))
	&=
	\scalarproduct[\big]{\G(u_\coarse'(\p))}{\G(v)}_{\CC}
	\quad
	\text{for all } v \in \XX,\\
	a(v_\coarse, z'_\coarse(\p))
	&=
	\scalarproduct[\big]{\G(u_\coarse'(\p))}{\G(v_\coarse)}_{\CC}
	\quad
	\text{for all } v_\coarse \in \XX_\coarse
\end{split}
\end{equation}
and note that the following identities hold:
\begin{equation}
\label{eq:zprime-combination}
	z'_\coarse(\p)
	=
	\sum_{j=1}^{n_\CC} G_j \big( u'_\coarse(\p) \big) z_j
	\quad \text{and} \quad
	z'_\coarse(\p)
	=
	\sum_{j=1}^{n_\CC} G_j \big( u'_\coarse(\p) \big) z_{\coarse,j}.
\end{equation}

\begin{remark}
	By considering the co-state components $z_{\coarse, j}$ from~\eqref{eq:def-dual-components}, the matrix entries $\B_{\coarse, ij}$ can be computed by the $n_\QQ$ state components $u_{\coarse, i}$, or the $n_\CC$ co-state components $z_{\coarse, j}$ via
	\begin{equation}
	\label{eq:discrete-lsmatrix}
		\B_{\coarse, ij}
		=
		G_j(u_{\coarse, i})
		\eqreff{eq:def-dual-components}{=}
		a(u_{\coarse,i}, z_j)
		=
		a(u_{\coarse,i}, z_{\coarse,j})
		\eqreff{eq:primal-components}{=}
		b(\e_i, z_{\coarse, j}).
	\end{equation}
	Thus, for assembling $\B_\coarse$, one can decide between computing state or co-state components.
	For our adaptive algorithm below, however, we need the state as well as the co-state components anyway to compute the necessary a posteriori estimators.
\end{remark}


\section{Adaptive algorithm and main results}\label{sec:results}

\subsection{A priori estimate}
Our first main result is an \textsl{a priori} estimate for the parameter error.
To the best of our knowledge, this is a novel result. 
Its proof is given in Section~\ref{sec:errorbound} below.
\begin{theorem}
	\label{th:parameter-apriori}
	There exists a constant $C_{\QQ} > 0$ such that
	\begin{equation}
	\label{eq:parameter-apriori}
		\norm{\p^\star - \p^\star_\coarse}_{\QQ}
		\leq
		C_{\QQ} \, \Big[ \sum_{i=0}^{n_\QQ} \enorm{u_i - u_{\coarse,i}}^2 \Big]^{1/2}
		\Big[ \sum_{j=1}^{n_\CC} \norm{z_j - z_{\coarse,j}}^2 \Big]^{1/2}
		\quad
		\text{for all } \TT_\coarse \in \T.
	\end{equation}
	The constant $C_{\QQ}$ depends only on $\QQ$, $\Omega$, $\G^\star$, $\kappa$, $\alpha$, and the data $\A$, $f_i$, $g_j$, $\f_i$, and $\g_j$.
\end{theorem}

\begin{remark}
	The proof of the estimate~\eqref{eq:parameter-apriori} relies heavily on the linear dependence of the primal and dual solution on the parameter, as well as the finite dimension of $\QQ$.
	Problems which depend on the parameter in a nonlinear fashion are usually solved iteratively, such that the linearization in each step again depends linearly on the parameter.
	Thus, a result analogous to Theorem~\ref{th:parameter-apriori} might also hold in this case.
	For optimal control problems with $\dim \QQ = \infty$, however, it is not clear how to replace the components $u_i$ and $z_j$ as well as their discretizations in~\eqref{eq:parameter-apriori}.
\end{remark}
The estimate~\eqref{eq:parameter-apriori} is the fundamental result that allows to design the adaptive algorithm that is presented in the following sections.

\subsection{Mesh refinement}
Let $\TT_0$ be a given conforming triangulation of $\Omega$ which is admissible for $d \geq 3$ in the sense of~\cite{stevenson08}.
For mesh refinement, we employ newest vertex bisection; see~\cite{bdd04,stevenson08,kpp13}.
For each conforming triangulation $\TT_\coarse$ and marked elements $\MM_\coarse \subseteq \TT_\coarse$, let $\TT_\fine := \refine(\TT_\coarse,\MM_\coarse)$ be the coarsest conforming triangulation where all $T \in \MM_\coarse$ have been refined, i.e., $\MM_\coarse \subseteq \TT_\coarse \backslash \TT_\fine$.
We write $\TT_\fine \in \T(\TT_\coarse)$, if $\TT_\fine$ results from $\TT_\coarse$ by finitely many steps of refinement.
To abbreviate notation, let $\T:=\T(\TT_0)$.
We note that there holds nestedness of finite element spaces, i.e., $\TT_\fine \in \T(\TT_\coarse)$ implies that $\XX_\coarse \subseteq \XX_\fine$.

We note that newest vertex bisection generates a family of shape-regular meshes, i.e., there exists a constant $\gamma > 0$ such that
\begin{equation*}
	\sup_{\TT \in \T} \underset{T \in \TT}{\max\vphantom{\sup}} \frac{\mathrm{diam}(T)^d}{|T|}
	\leq
	\gamma
	<
	\infty.
\end{equation*}

\subsection{A posteriori error estimation}
We consider standard residual error estimators, i.e., for $\TT_\coarse \in \T$, $T \in \TT_\coarse$, and $v_\coarse \in \XX_\coarse$ we define
\begin{equation*}
\begin{split}
	\eta_{\coarse,i}(T, v_\coarse)^2
	&:=
	h_T^2 \norm[\big]{f_i + \div( \f_i + \A \nabla v_\coarse )}_{L^2(T)}^2
	+ h_T \norm[\big]{\jump{( \f_i + \A \nabla v_\coarse ) \cdot \normalvec}}_{L^2(\partial T \cap \Omega)}^2,\\
	\zeta_{\coarse,j}(T, v_\coarse)^2
	&:=
	h_T^2 \norm[\big]{g_j + \div( \g_j + \A \nabla v_\coarse )}_{L^2(T)}^2
	+ h_T \norm[\big]{\jump{( \g_j + \A \nabla v_\coarse ) \cdot \normalvec}}_{L^2(\partial T \cap \Omega)}^2,
\end{split}
\end{equation*}
for all $0 \leq i \leq n_{\QQ}$ and $1 \leq j \leq n_{\CC}$, where $h_T := |T|^{1/d}$ and $\jump{\cdot}$ is the jump across element boundaries.
For a subset $\UU_\coarse \subseteq \TT_\coarse$, we further define
\begin{equation}
\label{eq:abbreviations-1}
	\eta_{\coarse,i}(\UU_\coarse, v_\coarse)^2
	:=
	\sum_{T \in \UU_\coarse} \eta_{\coarse,i}(T, v_\coarse)^2,
	\qquad
	\zeta_{\coarse,j}(\UU_\coarse, v_\coarse)^2
	:=
	\sum_{T \in \UU_\coarse} \zeta_{\coarse,j}(T, v_\coarse)^2,
\end{equation}
and we abbreviate
\begin{equation}
\label{eq:abbreviations-2}
\begin{split}
	\eta_{\coarse,i}(v_\coarse)
	&:=
	\eta_{\coarse,i}(\TT_\coarse, v_\coarse),
	\quad \hphantom{\zeta_{\coarse,j}}
	\eta_{\coarse,i}(\UU_\coarse)
	:=
	\eta_{\coarse,i}(\UU_\coarse, u_{\coarse,i}),
	\quad \hphantom{\zeta_{\coarse,j}}
	\eta_{\coarse,i}
	:=
	\eta_{\coarse,i}(\TT_\coarse),\\
	\zeta_{\coarse,j}(v_\coarse)
	&:=
	\zeta_{\coarse,j}(\TT_\coarse, v_\coarse),
	\quad \hphantom{\eta_{\coarse,i}}
	\zeta_{\coarse,j}(\UU_\coarse)
	:=
	\zeta_{\coarse,j}(\UU_\coarse, z_{\coarse,j}),
	\quad \hphantom{\eta_{\coarse,i}}
	\zeta_{\coarse,j}
	:=
	\zeta_{\coarse,j}(\TT_\coarse).
\end{split}
\end{equation}

It is well-known that, in our setting, these estimators satisfy the so-called \emph{axioms of adaptivity}; see, e.g., \cite{axioms}.

\begin{proposition}
\label{prop:axioms}
	There exist constants $\Cstab, \Crel, \Cdrel > 0$ and $0 < \qred < 1$ such that for all $\TT_\coarse \in \T(\TT_0)$, $\TT_\fine \in \T(\TT_\coarse)$, $0 \leq i \leq n_{\QQ}$, and $1 \leq j \leq n_{\CC}$ the residual error estimators satisfy the following properties:
	
	\renewcommand{\theenumi}{{A\arabic{enumi}}}
	\begin{enumerate}
		\bf
		\item\label{assumption:stab} Stability:
		\rm
		For all $v_\fine \in \XX_\fine$, $v_\coarse \in \XX_\coarse$, and $\UU_\coarse \subseteq \TT_\fine \cap \TT_\coarse$, it holds that
		\begin{equation*}
			\big| \eta_{\fine,i}(\UU_\coarse, v_\fine) - \eta_{\coarse,i}(\UU_\coarse, v_\coarse) \big|
			+ \big| \zeta_{\fine,j}(\UU_\coarse, v_\fine) - \zeta_{\coarse,j}(\UU_\coarse, v_\coarse) \big|
			\leq
			\Cstab \, \enorm{v_\fine - v_\coarse}.
		\end{equation*}
		
		\bf
		\item\label{assumption:red} Reduction:
		\rm
		For all $v_\coarse \in \XX_\coarse$, it holds that
		\begin{equation*}
			\eta_{\fine,i}(\TT_\fine \backslash \TT_\coarse, v_\coarse)
			\leq
			\qred \, \eta_{\coarse,i}(\TT_\coarse \backslash \TT_\fine, v_\coarse)
			\quad \text{ and } \quad
			\zeta_{\fine,j}(\TT_\fine \backslash \TT_\coarse, v_\coarse)
			\leq
			\qred \, \zeta_{\coarse,j}(\TT_\coarse \backslash \TT_\fine, v_\coarse).
		\end{equation*}
		
		\bf
		\item\label{assumption:rel} Reliability:
		\rm
		The state and co-state components $u_{\coarse,i}, z_{\coarse,j} \in \XX_\coarse$ satisfy that
		\begin{equation*}
			\enorm{u_i - u_{\coarse,i}}
			\leq
			\Crel \, \eta_{\coarse,i}
			\quad \text{ and } \quad
			\enorm{z_j - z_{\coarse,j}}
			\leq
			\Crel \, \zeta_{\coarse,j}.
		\end{equation*}
		
		\bf
		\item\label{assumption:drel} Discrete reliability:
		\rm
		The components $u_{\coarse,i}, z_{\coarse,j} \in \XX_\coarse$ and $u_{\fine,i}, z_{\fine,j} \in \XX_\fine$ satisfy
		\begin{equation*}
			\enorm{u_{\fine,i} - u_{\coarse,i}}
			\leq
			\Cdrel \, \eta_{\coarse,i}(\TT_\coarse \backslash \TT_\fine)
			\quad \text{ and } \quad
			\enorm{z_{\fine,j} - z_{\coarse,j}}
			\leq
			\Cdrel \, \zeta_{\coarse,j}(\TT_\coarse \backslash \TT_\fine).
			\tag*{\qed}
		\end{equation*}
	\end{enumerate}
\end{proposition}

Together with our \textsl{a priori} estimate Theorem~\ref{th:parameter-apriori}, reliability~\eqref{assumption:rel} immediately implies that a suitable combination of error estimators of the components is indeed an upper bound to the parameter error.
\begin{theorem}
\label{th:parameter-posteriori}
	With the constant $C_{\QQ} > 0$ from Theorem~\ref{th:parameter-apriori}, there holds
	\begin{equation}
	\label{eq:parameter-posteriori}
		\norm{\p^\star - \p^\star_\coarse}_{\QQ}
		\leq
		\Crel^2 \, C_{\QQ} \, \Big[ \sum_{i=0}^{n_\QQ} \eta_{\coarse,i}^2 \Big]^{1/2}
		\Big[ \sum_{j=1}^{n_\CC} \zeta_{\coarse,j}^2 \Big]^{1/2}
		\quad
		\text{for all } \TT_\coarse \in \T.
	\end{equation}
\end{theorem}

\begin{remark}
	The usual (residual) a posteriori estimate for optimal control problems in the literature is an upper bound for the sum
	\begin{equation*}
		\norm{\p^\star - \p^\star_\coarse}_{\QQ}
		+ \enorm{u(\p^\star_\coarse) - u_\coarse(\p^\star_\coarse)}
		+ \enorm{z(\p^\star_\coarse) - z_\coarse(\p^\star_\coarse)}.
	\end{equation*}
	Since the error in the parameter can be expected to be of higher order than that of the state and the co-state variable, such an estimate is sub-optimal with respect to the parameter error.
	Our estimate in~\eqref{eq:error-parameter} neatly exploits the finite dimension of $\QQ$ to gain two advantages:
	\begin{itemize}
		\item it shows an improved rate of convergence compared to the usual estimator for the sum above (see the numerical experiments in Section~\ref{sec:numerics});
		
		\item it does not use $\p^\star_\coarse$, so that its (possibly costly) computation can be avoided, if one is only interested in an a posteriori estimate.
		For instance, one can improve the mesh by use of Algorithm~\ref{algorithm} below until the upper bound in~\eqref{eq:parameter-posteriori} is sufficiently small.
		Then, $\p_\coarse^\star$ is only computed once for the final mesh.
	\end{itemize}
\end{remark}

\subsection{Adaptive algorithm}
Once Theorem~\ref{th:parameter-posteriori} provides an error estimator for the parameter error on a mesh $\TT_\coarse \in \T$ involving only (co-)state components, we can define weighted refinement indicators in the spirit of~\cite{bet11}:
\begin{equation}
\label{eq:weighted-estimator}
	\varrho_\coarse(T)^2
	:=
	\Big[ \sum_{i=0}^{n_{\QQ}} \eta_{\coarse,i}^2 \Big]
	\Big[ \sum_{j=1}^{n_{\CC}} \zeta_{\coarse,j}(T)^2 \Big]
	+
	\Big[ \sum_{i=0}^{n_{\QQ}} \eta_{\coarse,i}(T)^2 \Big]
	\Big[ \sum_{j=1}^{n_{\CC}} \zeta_{\coarse,j}^2 \Big].
\end{equation}
For a subset $\UU_\coarse \subseteq \TT_\coarse$, we again define $\varrho_\coarse(\UU_\coarse)^2 := \sum_{T \in \UU_\coarse} \varrho_{\coarse}(T)^2$ and $\varrho_\coarse := \varrho_\coarse(\TT_\coarse)$.

This allows us to devise an adaptive algorithm for the parameter estimation problem.
\begin{algorithm}
\label{algorithm}
	\textbf{Input:} Initial triangulation $\TT_0$, marking parameter $\theta \in (0,1]$.\\
	For $\ell = 0, 1, 2, \ldots$, do
	\begin{enumerate}[label={\rm (\roman*)}]
		\item Compute $u_{\ell,i}$ and $z_{\ell,j}$ for $i = 0, \ldots, n_\QQ$ and $j = 1, \ldots, n_\CC$.
		
		\item Compute refinement indicators $\varrho_\ell(T)$.
		
		\item Find a minimal set $\MM_\ell \subseteq \TT_\ell$ of elements such that $\varrho_\ell(\MM_\ell)^2 \geq \theta \varrho_\ell^2$.
		
		\item Compute $\TT_{\ell+1} := \refine(\TT_\ell, \MM_\ell)$.
	\end{enumerate}
	\textbf{Output:} Sequence of triangulations $(\TT_\ell)_\ell$, components $(u_{\ell,i})_\ell, (z_{\ell,j})_\ell$, and corresponding error estimators.
\end{algorithm}

\begin{remark}
	\label{rem:upper-bound}
	We note that the proposed weighted error estimator~\eqref{eq:weighted-estimator} satisfies
	\begin{equation}
	\label{eq:estimator-equivalence}
		\varrho_\coarse^2
		=
		2 \, \Big[ \sum_{i=0}^{n_{\QQ}} \eta_{\coarse,i}^2 \Big]
		\Big[ \sum_{j=1}^{n_{\CC}} \zeta_{\coarse,j}^2 \Big],
	\end{equation}
	which is essentially the square of the upper bound of the parameter error~\eqref{eq:parameter-posteriori} from Theorem~\ref{th:parameter-posteriori}.
	Thus, by proving convergence (with optimal rates) of Algorithm~\ref{algorithm} with respect to the weighted estimator $\varrho_\coarse$, we can draw the same conclusion for the parameter error.
\end{remark}

\begin{remark}
	Note that computing the parameter $\p^\star_\coarse$ from~\eqref{eq:first-order-condition-discrete} can involve high computational effort, depending on the precise constraints imposed by $\QQ$.
	For this reason, we stress that Algorithm~\ref{algorithm} relies only on the unweighted (co-)state components $u_{\coarse, i}$ and $z_{\coarse,j}$ for error estimation and refinement in each step.
	Therefore, $\p^\star_\coarse$ will only be computed once, namely when the upper bound in~\eqref{eq:parameter-posteriori} for $\TT_\coarse = \TT_\ell$ is sufficiently small and, hence, Algorithm~\ref{algorithm} is terminated for this step $\ell$.
\end{remark}

\begin{remark}
	Instead of the weighted marking strategy in Algorithm~\ref{algorithm}{\rm (iii)}, the marking strategies from~\cite{fpz16} and the seminal work~\cite{ms09} can also be used in our analysis below.
	Both strategies first find minimal sets $\MM_\ell^u$ and $\MM_\ell^z$ that satisfy
	\begin{align*}
		\theta \, \sum_{i=0}^{n_{\QQ}} \eta_{\ell,i}^2
		\leq
		\sum_{i=0}^{n_{\QQ}} \eta_{\ell,i}(\MM_\ell^u)^2
		\quad \text{and} \quad
		\theta \, \sum_{j=1}^{n_{\CC}} \zeta_{\ell,j}^2
		\leq
		\sum_{j=1}^{n_{\CC}} \zeta_{\ell,j}(\MM_\ell^z)^2,
	\end{align*}
	respectively.
	The set of marked elements is then defined as suitable subset $\MM_\ell \subseteq \MM_\ell^u \cup \MM_\ell^z$ with $\#\MM_\ell \leq \Cmark \min \{ \#\MM_\ell^u, \#\MM_\ell^z \}$, where $\Cmark = 1$ in~\cite{ms09} and $\Cmark = 2$ in~\cite{fpz16}.
	In particular, the minimality assumption to $\MM_\ell$, $\MM_\ell^u$, and $\MM_\ell^z$ can be weakened to only hold up to an arbitrary but fixed factor greater than $1$; see~\cite{fpz16}.
\end{remark}

\subsection{Convergence of Algorithm~\ref{algorithm}}
Our second main result concerns linear convergence of the error estimator.

\begin{theorem}
\label{th:linear-convergence}
	Suppose $0 < \theta \leq 1$.
	Then, Algorithm~\ref{algorithm} satisfies linear convergence
	\begin{equation}
	\label{eq:linear-convergence}
		\varrho_{\ell + n}
		\leq
		\Clin \qlin^{n} \, \varrho_{\ell}
		\quad \text{for all } \ell, n \in \N_0.
	\end{equation}
	The constants $\Clin > 0 $ and $0 < \qlin < 1$ depend only on $\Cstab$, $\qred$, $\Crel$, and the (arbitrary) adaptivity parameter $0 < \theta \leq 1$.
\end{theorem}

For our last main result, linear convergence of the error estimator $\varrho_\ell$ with optimal algebraic rates, we introduce so-called approximation classes.
Given $N \in \N_0$, let $\T(N)$ be the set of all $\TT_\coarse \in \T$ with $\#\TT_\coarse - \#\TT_0\le N$.
For all $r>0$, we define
\begin{equation}
\label{eq:approximation-classes}
\begin{split}
	\norm{u_i}_{\mathbb{A}_r}
	&:=
	\sup_{N \in \N_0}
	(N+1)^r \min_{\TT_{\rm opt} \in \T(N)} \eta_{{\rm opt},i}(u_{{\rm opt},i})
	\in [0,\infty],
	\quad
	\text{for } 0 \leq i \leq n_{\QQ},\\
	\norm{z_j}_{\mathbb{A}_r}
	&:=
	\sup_{N \in \N_0}
	(N+1)^r \min_{\TT_{\rm opt} \in \T(N)} \zeta_{{\rm opt},j}(z_{{\rm opt},j})
	\in [0,\infty],
	\quad
	\text{for } 1 \leq j \leq n_{\CC}.
\end{split}
\end{equation}
By definition, e.g., $\norm{u_i}_{\mathbb{A}_r} < \infty$ yields that $\eta_{{\rm opt},i}(u_{{\rm opt},i})$ decays at least with algebraic rate $r > 0$ along a sequence of optimal meshes.
The following theorem states that any possible overall rate for the upper bound of~\eqref{eq:parameter-posteriori} will indeed be realized by Algorithm~\ref{algorithm}.

\begin{theorem}
\label{th:optimal-rates}
	Let $0 < \theta < \theta_\opt := (1+\Cstab^2\Cdrel^2)^{-1}$.
	Let $s_i, t_j > 0$ with
	\begin{equation*}
		\sum_{i=0}^{n_{\QQ}} \norm{u_i}_{\mathbb{A}_{s_i}}^2
		+ \sum_{j=1}^{n_{\CC}} \norm{z_j}_{\mathbb{A}_{t_j}}^2
		< \infty.
	\end{equation*}
	Then, there exists a constant $\Copt>0$ such that for all $\ell \in \N_0$ there holds that
	\begin{align}
	\label{eq:optimal-rates}
		\varrho_\ell
		\leq
		\Copt \,
		\Big[ \sum_{i=0}^{n_{\QQ}} \norm{u_i}_{\mathbb{A}_{s_i}}^2 \Big]^{1/2}
		\Big[ \sum_{j=1}^{n_{\CC}} \norm{z_j}_{\mathbb{A}_{t_j}}^2 \Big]^{1/2}
		\big( \#\TT_\ell - \#\TT_0 \big)^{-\beta},
	\end{align}
	where $\beta := \min \set{s_i}{0 \leq i \leq n_{\QQ}} + \min \set{t_j}{1 \leq j \leq n_{\CC}}$.
	The constant $\Copt$ depends only on $\Ccls$, $\Cstab$, $\qred$, $\Crel$, $\Cdrel$, $\theta$, $n_\QQ$, $n_\CC$, $s_i$, and $t_j$.
\end{theorem}

\begin{remark}
	Note that our adaptive algorithm drives down only the upper bound of the parameter error with optimal rates and not the parameter error itself.
	However, in general, one cannot expect to obtain a rate higher than $\min \set{s_i}{0 \leq i \leq n_{\QQ}} + \min \set{t_j}{1 \leq j \leq n_{\CC}}$ for the parameter error.
	This is due to the fact that the assembly of the matrix entries $\B_{\ell, ij}$, which are needed to compute $\p_\ell^\star$, uses all state or co-state components and thus its accuracy is constricted by their respective minimal rate.
\end{remark}

\begin{remark}
	Theorem~\ref{th:optimal-rates} holds true for all mesh refinement strategies as long as there hold the \emph{son estimate}
	\begin{equation}
	\label{eq:mesh-sons}
		\#(\TT_\coarse \backslash \TT_\fine) + \#\TT_\coarse
		\leq
		\#\TT_\fine
		\quad \text{for all } \TT_\coarse \in \T
		\text{ and all } \TT_\fine \in \T(\TT_\coarse),
	\end{equation}
	the \emph{overlay estimate}
	\begin{equation}
	\label{eq:mesh-overlay}
		\#(\TT_\coarse \oplus \TT_\fine)
		\leq
		\#\TT_\coarse + \#\TT_\fine - \#\TT_0,
		\quad \text{for all } \TT_\coarse \in \T
		\text{ and all } \TT_\fine \in \T(\TT_\coarse),
	\end{equation}
	and the \emph{closure estimate} (with constant $\Ccls > 0$)
	\begin{equation}
	\label{eq:mesh-closure}
		\#\TT_\ell - \#\TT_0
		\leq
		\Ccls \, \sum_{j=0}^{\ell-1} \#\MM_j
		\quad \text{for all } \ell \in \N,
	\end{equation}
	as well as the axioms~\eqref{assumption:stab}--\eqref{assumption:drel}; see~\cite{axioms}.
	In this sense, the present analysis is indeed independent of newest vertex bisection.
\end{remark}


\section{Proof of Theorem~\ref{th:parameter-apriori}}\label{sec:errorbound}

\subsection{Auxiliary a priori bounds}
We start by stating some well-known \textsl{a priori} estimates, which are used throughout the subsequent analysis.
\begin{lemma}
	\label{lemma:apriori-uz}
	There exists a constant $\Cproblem > 0$ such that, for all $\TT_\coarse \in \T$,
	\begin{equation}
	\label{eq:apriori-uz}
	\begin{split}
		\enorm{u_{\coarse,i}} \leq \enorm{u_i}
		&\leq
		\Cproblem \, \big[ \norm{f_i}_{L^2(\Omega)} + \norm{\f_i}_{L^2(\Omega)} \big]
		\quad
		\text{for all } 0 \leq i \leq n_{\QQ},\\
		\enorm{z_{\coarse,j}} \leq \enorm{z_j}
		&\leq
		\Cproblem \, \big[ \norm{g_j}_{L^2(\Omega)} + \norm{\g_j}_{L^2(\Omega)} \big]
		\quad
		\text{for all } 1 \leq j \leq n_{\CC}.
	\end{split}
	\end{equation}
	The constant $\Cproblem$ depends only on $\A$ and $\Omega$.
	\hfill$\square$
\end{lemma}

Since $u'(\p)$, $z'(\p)$, and their discrete counterparts depend linearly on the parameter, we can use the \textsl{a priori} bounds to split parameters and component errors. 

\begin{lemma}
\label{lemma:derivative-estimate}
	There exists a constant $C > 0$ such that, for all $\p \in \QQ$,
	\begin{equation}
	\label{eq:derivative-estimate}
	\begin{split}
		\enorm[\big]{u'(\p) - u'_\coarse(\p)}
		&\leq
		\norm{\p}_{\QQ} \Big[ \sum_{i=1}^{n_\QQ} \enorm{u_i - u_{\coarse,i}}^2 \Big]^{1/2},\\
		\enorm[\big]{z'(\p) - z'_\coarse(\p)}
		&\leq
		C \, \norm{\p}_{\QQ} \Big[ \sum_{j=1}^{n_\CC} \enorm{z_j - z_{\coarse,j}}^2 \Big]^{1/2}.
	\end{split}
	\end{equation}
	The constant $C$ depends only on $\Omega$ as well as the data $\A$, $f_i$, $g_j$, $\f_i$, and $\g_j$.
\end{lemma}

\begin{proof}
	Let $\p \in \QQ$.
	With the representation~\eqref{eq:component-combination}, the triangle inequality, and the Cauchy--Schwarz inequality, we obtain the first inequality of~\eqref{eq:derivative-estimate} by
	\begin{equation*}
		\enorm{u'(\p) - u'_\coarse(\p)}
		\eqreff{eq:component-combination}{\leq}
		\sum_{i=1}^{n_\QQ} |\p_i| \enorm{u_i - u_{\coarse,i}}
		\leq
		\norm{\p}_{\QQ} \Big[ \sum_{i=1}^{n_\QQ} \enorm{u_i - u_{\coarse,i}}^2 \Big]^{1/2}.
	\end{equation*}
	The same arguments, together with  a priori bounds for $u_{\coarse,i}$, can be used to show that
	\begin{equation}
	\label{eq:Du-bounded}
		\enorm{u'_\coarse(\p)}
		\eqreff{eq:component-combination}{\leq}
		\norm{\p}_{\QQ} \Big[ \sum_{i=1}^{n_\QQ} \enorm{u_{\coarse,i}}^2 \Big]^{1/2}
		\, \eqreff{eq:apriori-uz}{\leq}
		\norm{\p}_{\QQ} \Cproblem \Big[ \sum_{i=1}^{n_\QQ} \big( \norm{f_i}_{L^2(\Omega)} + \norm{\f_i}_{L^2(\Omega)} \big)^2 \Big]^{1/2}.
	\end{equation}
	For the second inequality in~\eqref{eq:derivative-estimate}, we can again employ the triangle inequality.
	Together with continuity of the measurement functionals $G_j$, which depend only on $g_j$ and $\g_j$, and~\eqref{eq:Du-bounded} this leads to
	\begin{align*}
		\enorm[\big]{z'(\p) - z'_\coarse(\p)}
		&\eqreff{eq:zprime-combination}{\leq}
		\sum_{j=1}^{n_\CC} \big| G_j \big( u'_\coarse(\p) \big) \big| \, \enorm{z_j - z_{\coarse,j}}\\
		&\leq
		\norm[\big]{\G \big( u'_\coarse(\p) \big)}_{\CC} \Big[ \sum_{j=1}^{n_\CC} \enorm{z_j - z_{\coarse,j}}^2 \Big]^{1/2}
		\, \eqreff{eq:Du-bounded}{\lesssim}
		\norm{\p}_{\QQ} \Big[ \sum_{j=1}^{n_\CC} \enorm{z_j - z_{\coarse,j}}^2 \Big]^{1/2}.
	\end{align*}
	This concludes the proof.
\end{proof}

\begin{lemma}
	\label{lemma:apriori-p}
	There exists a constant $C_\star > 0$ such that, for all $\TT_\coarse \in \T$,
	\begin{equation}
	\label{eq:apriori-p}
		\norm{\p^\star}_{\QQ} \leq C_{\star}
		\quad \text{ and } \quad
		\norm{\p^\star_\coarse}_{\QQ} \leq C_{\star}.
	\end{equation}
	The constant $C_\star$ depends only on $\alpha$, $\G^\star$, $\Omega$, $\kappa$, $\min \set{\norm{\q}_{\QQ}}{\q \in \QQ}$, and the data $A$, $f_i$, $g_j$, $\f_i$, and $\g_j$.
\end{lemma}

\begin{proof}
	We first define $C := \min \set{\norm{\q}_{\QQ}}{\q \in \QQ}$ and choose $\p \in \QQ$ such that $\norm{\p}_{\QQ} = C$ (such a choice exists since $\QQ$ is closed and convex).
	From the first and second order optimality condition, \eqref{eq:first-order-condition} and \eqref{eq:locally-convex}, and the explicit form of $J'$ in~\eqref{eq:Jprime-explicit}, we have that
	\begin{align*}
		\kappa \, \norm{\p^\star}_{\QQ}^2
		&\eqreff{eq:locally-convex}{\leq}
		J''(\p^\star, \p^\star)
		\eqreff{eq:hessian-explicit}{=}
		(\p^\star)^\intercal \big( \B \B^\intercal + \alpha \I \big) \p^\star
		\eqreff{eq:Jprime-explicit}{=}
		J'[\p^\star](\p^\star) - (\p^\star)^\intercal \B \big( \G(u_0) - \G^\star \big)\\
		&\qquad \eqreff{eq:first-order-condition}{\leq}
		J'[\p^\star](\p) - (\p^\star)^\intercal \B \big( \G(u_0) - \G^\star \big)\\
		&\qquad \eqreff{eq:Jprime-explicit}{=}
		(\p^\star)^\intercal \big( \B \B^\intercal + \alpha \I \big) \p
		+ (\p - \p^\star)^\intercal \B \big( \G(u_0) - \G^\star \big).
	\end{align*}
	We denote by $\norm{\cdot}_{L(\QQ)}$ the natural matrix norm induced by the Euclidean norm $\norm{\cdot}_{\QQ}$, i.e., the spectral norm.
	Using the Cauchy--Schwarz inequality together with $\norm{\p}_{\QQ} = C$, we see that
	\begin{equation*}
		\kappa \, \norm{\p^\star}_{\QQ}^2
		\leq
		\norm{\p^\star}_{\QQ}
		\big(
			C \, \norm{\B \B^\intercal + \alpha \I}_{L(\QQ)} + \norm{\B ( \G(u_0) - \G^\star )}_{\QQ}
		\big)
		+ C \, \norm{\B ( \G(u_0) - \G^\star )}_{\QQ}.
	\end{equation*}
	In the case $\norm{\p^\star}_{\QQ} < 1$, there already holds the first inequality of~\eqref{eq:apriori-p} with $C_\star = 1$.
	In the case $\norm{\p^\star}_{\QQ} \geq 1$, we can divide by $\kappa \norm{\p^\star} \geq \kappa > 0$ and further estimate the last inequality by
	\begin{equation}
	\label{eq:apriori-p-1}
		\norm{\p^\star}_{\QQ}
		\leq
		\kappa^{-1}
		\big(
			C \, \norm{\B \B^\intercal + \alpha \I}_{L(\QQ)}
			+ (1 + C ) \norm{\B ( \G(u_0) + \G^\star )}_{\QQ}
		\big).
	\end{equation}
	From the definition of $\B$, we have that, for all $1 \leq i \leq n_{\QQ}$ and $1 \leq j \leq n_{\CC}$,
	\begin{equation*}
		|\B_{ij}|
		=
		|G_j(u_i)|
		=
		a(u_i, z_j)
		\eqreff{eq:apriori-uz}{\leq}
		\Cproblem^2 \,
		\big[ \norm{f_i}_{L^2(\Omega)} + \norm{\f_i}_{L^2(\Omega)} \big]
		\big[ \norm{g_j}_{L^2(\Omega)} + \norm{\g_j}_{L^2(\Omega)} \big].
	\end{equation*}
	We can thus estimate the Frobenius-norm $\norm{\cdot}_F$ of $B$ by
	\begin{equation*}
		\norm{B}_{F}^2
		=
		\sum_{i=1}^{n_{\QQ}} \sum_{j=1}^{n_{\CC}} |B_{ij}|^2
		\leq
		\Cproblem^4 \sum_{i=1}^{n_{\QQ}} \sum_{j=1}^{n_{\CC}} 
		\big[ \norm{f_i}_{L^2(\Omega)} + \norm{\f_i}_{L^2(\Omega)} \big]^2
		\big[ \norm{g_j}_{L^2(\Omega)} + \norm{\g_j}_{L^2(\Omega)} \big]^2.
	\end{equation*}
	Using these estimates together with
	\begin{equation*}
		\norm{\B \B^\intercal + \alpha \I}_{L(\QQ)}
		\leq
		\norm{B}_F^2 + \alpha
		\quad \text{ and } \quad
		\norm{\B ( \G(u_0) - \G^\star )}_{\QQ}
		\leq
		\norm{\B}_{F} \norm{\G(u_0) - \G^\star}_{\CC}
	\end{equation*}
	to bound the matrix norms in~\eqref{eq:apriori-p-1}, we show the continuous estimate in~\eqref{eq:apriori-p}.
	Since the estimate of $\B_{ij}$ holds in the discrete case as well, the discrete estimate in~\eqref{eq:apriori-p} follows analogously.
\end{proof}

The next lemma shows an estimate similar to Lemma~\ref{lemma:derivative-estimate}, but without the additional factor $\norm{\p}_{\QQ}$.
Note that such an additional factor cannot be expected since neither $u(\p)$ nor $z(\p)$ are linear in the parameter.
\begin{lemma}
	\label{lemma:uz-estimate}
	There exists a constant $C > 0$ such that, for all $\TT_\coarse \in \T$ and $\p \in \{ \p^\star, \p^\star_\coarse \}$,
	\begin{equation}
	\label{eq:uz-estimate}
	\begin{split}
		\enorm[\big]{u(\p) - u_\coarse(\p)}
		&\leq
		C \Big[ \sum_{i=0}^{n_\QQ} \enorm{u_i - u_{\coarse,i}}^2 \Big]^{1/2},\\
		\quad
		\enorm[\big]{z(\p) - z_\coarse(\p)}
		&\leq
		C \Big[ \sum_{j=1}^{n_\CC} \enorm{z_j - z_{\coarse,j}}^2 \Big]^{1/2}.
	\end{split}
	\end{equation}
	The constant $C$ depends only on $\G^\star$, the data $\A$, $f_i$, $g_j$, $\f_i$, and $\g_j$, and the constants $\Cproblem$ from Lemma~\ref{lemma:apriori-uz} and $C_\star$ from Lemma~\ref{lemma:apriori-p}.
\end{lemma}

\begin{proof}
	The estimate for the difference in the state $u$ follows along the same lines as in the proof of Lemma~\ref{lemma:derivative-estimate}, where the parameter can be estimated by~\eqref{eq:apriori-p}.
	For the difference in the co-state, we note that~\eqref{eq:apriori-uz} and $\p \in \{ \p^\star, \p^\star_\coarse \}$ imply that
	\begin{equation}
	\label{eq:uz-estimate-1}
		\enorm{u_\coarse(\p)} \leq C < \infty,
	\end{equation}
	where the constant $C > 0$ depends only on $\Cproblem$ from Lemma~\ref{lemma:apriori-uz}, the data $f_i$ and $\f_i$, and $C_\star$ from Lemma~\ref{lemma:apriori-p}.
	Thus, with continuity of the measurement functionals $G_j$, which depend only on $g_j$, $\g_j$, we have that
	\begin{align*}
		\enorm[\big]{z(\p) - z_\coarse(\p)}
		&\leq
		\sum_{j=1}^{n_\CC} \big| G_j \big( u_\coarse(\p) \big) - G^\star_j \big| \, \enorm{z_j - z_{\coarse,j}}\\
		&\leq
		\norm[\big]{\G \big( u_\coarse(\p) \big) - \G^\star}_{\CC} \Big[ \sum_{j=1}^{n_\CC} \enorm{z_j - z_{\coarse,j}}^2 \Big]^{1/2}\\
		&\lesssim
		\big( \enorm{u_\coarse(\p)} + \norm{\G^\star}_{\CC} \big) \Big[ \sum_{j=1}^{n_\CC} \enorm{z_j - z_{\coarse,j}}^2 \Big]^{1/2}
		\, \eqreff{eq:uz-estimate-1}{\lesssim}
		\Big[ \sum_{j=1}^{n_\CC} \enorm{z_j - z_{\coarse,j}}^2 \Big]^{1/2}.
	\end{align*}
	This concludes the proof.
\end{proof}

\subsection{Error bound for parameter error}
The next lemma estimates the error in the derivative of the least-squares functional.

\begin{lemma}
	\label{lemma:error-Jderivative}
	For all $\q, \p \in \QQ$, it holds that
	\begin{equation}
	\label{eq:error-Jderivative}
	\begin{split}
		|J'[\p](\q) - J'_\coarse[\p](\q)|
		&\leq
		\enorm[\big]{u(\p) - u_\coarse(\p)} \enorm[\big]{z'(\q) - z'_\coarse(\q)}\\
		&\quad
		+ \enorm[\big]{u'(\q) - u'_\coarse(\q)} \enorm[\big]{z(\p) - z_\coarse(\p)}\\
		&\quad
		+ \enorm[\big]{u(\p) - u_\coarse(\p)}
		\enorm[\big]{u'(\q) - u'_\coarse(\q)}
		\sum_{j=1}^{n_\CC} \enorm[\big]{z_j - z_{\coarse,j}}^2.
	\end{split}
	\end{equation}
\end{lemma}

\begin{proof}
	Let $\q, \p \in \QQ$.
	Note that $\r'(\q) = \G(u'(\q))$ and $\r'_\coarse(\q) = \G(u'_\coarse(\q))$.
	Therefore, we have that
	\begin{align}
	\nonumber
		J'[\p](\q) - &J'_\coarse[\p](\q)
		\eqreff{eq:Jprime}{=}
		\scalarproduct[\big]{\r(\p)}{\r'(\q)}
		- \scalarproduct[\big]{\r_\coarse(\p)}{\r'_\coarse(\q)}\\
	\label{eq:error-Jderivative-1}
		&=
		\scalarproduct[\big]{\r(\p) - \r_\coarse(\p)}{\r'(\q)}
		+ \scalarproduct[\big]{\r_\coarse(\p)}{\r'(\q) - \r'_\coarse(\q)}\\
	\nonumber
		&=
		\scalarproduct[\Big]{\r(\p) - \r_\coarse(\p)}{\G(u'(\q))}
		+ \scalarproduct[\Big]{\r_\coarse(\p)}{\G\big( u'(\q) - u'_\coarse(\q) \big)}.
	\end{align}
	The first term in~\eqref{eq:error-Jderivative-1} can be reformulated by inserting the term $\G( u'_\coarse(\q))$, the definition of $z'(\q)$, and the Galerkin orthogonality.
	This yields that
	\begin{align*}
		&\scalarproduct[\big]{\r(\p) - \r_\coarse(\p)}{\G(u'(\q))}\\
		&\quad =
		\scalarproduct[\Big]{\G \big(u(\p) - u_\coarse(\p) \big)}{\G\big( u'(\q) - u'_\coarse(\q) \big)}
		+ \scalarproduct[\Big]{\G \big(u(\p) - u_\coarse(\p) \big)}{\G\big( u'_\coarse(\q) \big)}\\
		&\quad \eqreff{eq:zprime-equation}{=}
		\scalarproduct[\Big]{\G \big(u(\p) - u_\coarse(\p) \big)}{\G\big( u'(\q) - u'_\coarse(\q) \big)}
		+ a \big( u(\p) - u_\coarse(\p), z'(\q) \big)\\
		&\quad =
		\scalarproduct[\Big]{\G \big(u(\p) - u_\coarse(\p) \big)}{\G\big( u'(\q) - u'_\coarse(\q) \big)}
		+ a \big( u(\p) - u_\coarse(\p), z'(\q) - z'_\coarse(\q) \big).
	\end{align*}
	We employ the definition of the co-state components $z_j$ and the Galerkin orthogonalities to obtain that
	\begin{align*}
		\scalarproduct[\big]{\G \big(u(\p) - u_\coarse(\p) \big)&}{\G\big( u'(\q) - u'_\coarse(\q) \big)}
		=
		\sum_{j=1}^{n_\CC} G_j \big( u(\p) - u_\coarse(\p) \big) \,
		G_j \big( u'(\q) - u'_\coarse(\q) \big)\\
		&\eqreff{eq:def-dual-components}{=}
		\sum_{j=1}^{n_\CC} a \big( u(\p) - u_\coarse(\p), z_j \big) \,
		a \big( u'(\q) - u'_\coarse(\q), z_j \big)\\
		&=
		\sum_{j=1}^{n_\CC} a \big( u(\p) - u_\coarse(\p), z_j - z_{\coarse,j} \big) \,
		a \big( u'(\q) - u'_\coarse(\q), z_j - z_{\coarse,j} \big).
	\end{align*}
	For the second term in~\eqref{eq:error-Jderivative-1}, we use the definition of $z(\q)$ and the Galerkin orthogonality to obtain that
	\begin{align*}
		\scalarproduct[\Big]{\r_\coarse(\p)}{\G\big( u'(\q) - u'_\coarse(\q) \big)}
		&\eqreff{eq:z-equation}{=}
		a \big( u'(\q) - u'_\coarse(\q), z(\p) \big)\\
		&=
		a \big( u'(\q) - u'_\coarse(\q), z(\p) - z_\coarse(\p) \big).
	\end{align*}
	Finally, the claim~\eqref{eq:error-Jderivative} results from combining above identities and using the Cauchy--Schwarz inequality.
\end{proof}

Finally, we combine the last auxiliary results to obtain an estimate for the error in the parameter.

\begin{lemma}
	\label{lemma:error-parameter}
	There exists a constant $C > 0$ such that the approximation error in the parameter can be estimated by
	\begin{align}
	\nonumber
		\norm{\p^\star - \p^\star_\coarse}_{\QQ}
		&\leq
		C \Big[ 
			\enorm[\big]{u(\p^\star_\coarse) - u_\coarse(\p^\star_\coarse)}
			\big[ \sum_{j=1}^{n_\CC}  \enorm{z_j - z_{\coarse,j}}^2 \big]^{1/2}\\
	\label{eq:error-parameter}		
		& \qquad \qquad \times
			\Big(
				1 + \big[ \sum_{j=1}^{n_\CC}  \enorm{z_j - z_{\coarse,j}}^2 \big]^{1/2}
				\big[ \sum_{i=1}^{n_\QQ} \enorm{u_i - u_{\coarse,i}}^2 \big]^{1/2}
			\Big)\\
	\nonumber
		& \qquad +
			\enorm[\big]{z(\p^\star_\coarse) - z_\coarse(\p^\star_\coarse)}
			\big[ \sum_{i=1}^{n_\QQ} \enorm{u_i - u_{\coarse,i}}^2 \big]^{1/2}
		\Big].
	\end{align}
	The constant $C > 0$ depends only on $\Omega$ and $\kappa$ as well as the data $\A$, $f_i$, $g_j$, $\f_i$, and $\g_j$.
\end{lemma}

\begin{proof}
	In the following, let $\q := \p^\star_\coarse - \p^\star$.
	Note that the second order optimality condition~\eqref{eq:locally-convex} holds for all $\q \in \R^{n_\QQ}$, since $J''$ is independent of its linearization point.
	Therefore, we have that
	\begin{equation*}
		\kappa \norm{\q}_{\QQ}^2
		\eqreff{eq:locally-convex}{\leq}
		J''(\q,\q)
		\stackrel{\eqref{eq:Jprime-explicit}, \eqref{eq:hessian-explicit}}{=}
		J'[\p^\star_\coarse](\q) - J'[\p^\star](\q),
	\end{equation*}
	since $J'[\cdot](\q)$ is affine.
	With the continuous and discrete first order optimality conditions, \eqref{eq:first-order-condition} and~\eqref{eq:first-order-condition-discrete}, respectively, we see that
	\begin{align*}
		\kappa \norm{\q}_{\QQ}^2
		&\leq
		\big( J'[\p^\star_\coarse](\q) - J'_\coarse[\p^\star_\coarse](\q) \big)
		+ \big( J'_\coarse[\p^\star_\coarse](\q) - J'[\p^\star](\q) \big)\\
		&\leq
		J'[\p^\star_\coarse](\q) - J'_\coarse[\p^\star_\coarse](\q).
	\end{align*}
	This last expression can be further bounded by Lemma~\ref{lemma:error-Jderivative}:
	\begin{align*}
		\kappa \norm{\q}_{\QQ}^2
		&\eqreff{eq:error-Jderivative}{\leq}
		\enorm[\big]{u(\p^\star_\coarse) - u_\coarse(\p^\star_\coarse)}
		\enorm[\big]{z'(\q) - z_\coarse'(\q)}
		+ \enorm[\big]{u'(\q) - u_\coarse'(\q)}
		\enorm[\big]{z(\p^\star_\coarse) - z_\coarse(\p^\star_\coarse)}\\
		&\quad
		+ \enorm[\big]{u(\p^\star_\coarse) - u_\coarse(\p^\star_\coarse)}
		\enorm[\big]{u'(\q) - u_\coarse'(\q)}
		\sum_{j=1}^{n_\CC} \enorm[\big]{z_j - z_{\coarse,j}}^2.
	\end{align*}	
	From the right-hand side, a factor $\norm{\q}_{\QQ}$ can be split off from the $\q$-dependent terms by Lemma~\ref{lemma:derivative-estimate} to obtain that
	\begin{align*}
		\kappa \norm{\q}_{\QQ}^2
		&\eqreff{eq:derivative-estimate}{\lesssim}
		\enorm[\big]{u(\p^\star_\coarse) - u_\coarse(\p^\star_\coarse)}
		\norm{\q}_{\QQ} \Big[ \sum_{j=1}^{n_{\CC}} \enorm[\big]{z_j - z_{\coarse,j}}^2 \Big]^{1/2}\\
		&\quad
		+ \norm{\q}_{\QQ} \Big[ \sum_{i=1}^{n_{\QQ}} \enorm[\big]{u_i - u_{\coarse,i}}^2 \Big]^{1/2} \enorm[\big]{z(\p^\star_\coarse) - z_\coarse(\p^\star_\coarse)}\\
		&\quad
		+ \enorm[\big]{u(\p^\star_\coarse) - u_\coarse(\p^\star_\coarse)}
		\norm{\q}_{\QQ} \Big[ \sum_{i=1}^{n_{\QQ}} \enorm[\big]{u_i - u_{\coarse,i}}^2 \Big]^{1/2}
		\sum_{j=1}^{n_\CC} \enorm[\big]{z_j - z_{\coarse,j}}^2.
	\end{align*}
	The claim follows by division through $\kappa \norm{\q}_{\QQ} = \kappa \norm{\p^\star - \p^\star_\coarse}_{\QQ}$.
\end{proof}

From~\eqref{eq:error-parameter}, we can absorb higher order terms to deduce an upper bound for the parameter error, which is the assertion of Theorem~\ref{th:parameter-apriori}.

\begin{proof}[Proof of Theorem~\ref{th:parameter-apriori}]
	Note that due to the stability estimate~\eqref{eq:apriori-uz} there holds for all $0 \leq i \leq n_{\QQ}$ and all $1 \leq j \leq n_{\CC}$ that
	\begin{equation*}
		\enorm{u_i - u_{\coarse,i}}
		\leq
		2\Cproblem \, \big[ \norm{f_i}_{L^2(\Omega)} + \norm{\f_i}_{L^2(\Omega)} \big],
		\quad
		\enorm{z_j - z_{\coarse,j}}
		\leq
		2\Cproblem \, \big[ \norm{g_j}_{L^2(\Omega)} + \norm{\g_j}_{L^2(\Omega)} \big].
	\end{equation*}
	Using these estimates on the factor in~\eqref{eq:error-parameter}, we obtain that
	\begin{equation}
	\label{eq:rough-estimate}
	\begin{split}
		1 + \Big[ \sum_{j=1}^{n_\CC}  \enorm{z_j - z_{\coarse,j}}^2 \Big]^{1/2}
		\Big[ \sum_{i=1}^{n_\QQ} \enorm{u_i - u_{\coarse,i}}^2 \Big]^{1/2}
		\lesssim
		1,
	\end{split}
	\end{equation}
	where the hidden constant depends only on $\Cproblem$ and the data $f_i$, $g_j$, $\f_i$, and $\g_j$.
	Hence, \eqref{eq:error-parameter} reads as
	\begin{align*}
		\norm{\p^\star - \p^\star_\coarse}_{\QQ}
		&\lesssim
		\enorm[\big]{u(\p^\star_\coarse) - u_\coarse(\p^\star_\coarse)}
		\Big[ \sum_{j=1}^{n_\CC}  \enorm{z_j - z_{\coarse,j}}^2 \Big]^{1/2}\\
		&\qquad
		+ \enorm[\big]{z(\p^\star_\coarse) - z_\coarse(\p^\star_\coarse)}
		\Big[ \sum_{i=1}^{n_\QQ} \enorm{u_i - u_{\coarse,i}}^2 \Big]^{1/2}.
	\end{align*}
	Lemma~\ref{lemma:uz-estimate} allows to bound the energy norms to finally obtain~\eqref{eq:parameter-apriori}.
	This concludes the proof.
\end{proof}

\begin{remark}
	Note that our adaptive Algorithm~\ref{algorithm} guarantees convergence
	\begin{equation*}
		\Big[ \sum_{j=1}^{n_\CC}  \enorm{z_j - z_{\coarse,j}}^2 \Big]^{1/2}
		\Big[ \sum_{i=0}^{n_\QQ} \enorm{u_i - u_{\coarse,i}}^2 \Big]^{1/2}
		\to 0
		\quad
		\text{as } \ell \to \infty.
	\end{equation*}
	In particular, this implies that the estimate~\eqref{eq:rough-estimate} is too pessimistic, as the estimated term asymptotically tends to $1$.
\end{remark}


\section{Proof of Theorems~\ref{th:linear-convergence} and~\ref{th:optimal-rates}}\label{sec:optimalrates}

\subsection{Linear convergence}
We aim to employ the analysis of goal-oriented AFEM done by \cite{fgh+16,fpz16} to prove convergence rates with optimal algebraic rates for the error estimator.
To this end, we first note that the axioms presented in Proposition~\ref{prop:axioms} also hold for the sums of state and co-state components, respectively.
For convenience of the reader, we state the main intermediate results for proving Theorems~\ref{th:linear-convergence} and~\ref{th:optimal-rates}.

In view of Algorithm~\ref{algorithm}, we define estimators $\widetilde{\eta}_\coarse$ on $\XX_\coarse^{n_\QQ+1}$ and $\widetilde{\zeta}_\coarse$ on $\XX_\coarse^{n_\CC}$ by
\begin{equation}
\label{eq:tilde-estimators}
\begin{split}
	\widetilde{\eta}_\coarse(T, \v_\coarse)^2
	&:=
	\sum_{i=0}^{n_{\QQ}} \eta_{\coarse,i}(T, v_{\coarse,i})^2
	\quad
	\text{for all } \v_\coarse \in \XX_\coarse^{n_\QQ+1},\\
	\widetilde{\zeta}_\coarse(T, \w_\coarse)^2
	&:=
	\sum_{j=1}^{n_{\CC}} \zeta_{\coarse,j}(T, w_{\coarse,j})^2
	\quad
	\text{for all } \w_\coarse \in \XX_\coarse^{n_\CC}\\
\end{split}
\end{equation}
such that there holds
\begin{equation}
\label{eq:estimator-identity}
	\varrho_\coarse(T)^2
	=
	2 \, \widetilde{\eta}_\coarse(T, \u_\coarse)^2 \widetilde{\zeta}_\coarse(T, \z_\coarse)^2,
\end{equation}
where we set $\u_\coarse = (u_{\coarse,i})_{i=0}^{n_\QQ} \in \XX_\coarse^{n_\QQ+1}$ and $\z_\coarse = (z_{\coarse,j})_{j=1}^{n_\CC} \in \XX_\coarse^{n_\CC}$.
We employ the same notation, e.g., $\widetilde{\eta}_\coarse(\UU_\coarse, \v_\coarse)$ or $\widetilde{\zeta}_\coarse(\w_\coarse)$, as for $\eta_{\coarse,i}$ and $\zeta_{\coarse,j}$ in~\eqref{eq:abbreviations-1}--\eqref{eq:abbreviations-2}.
Moreover, we equip the spaces $\XX^{n_\QQ+1}$ and $\XX^{n_\CC}$ with the norms
\begin{equation*}
	\enorm{\v}_{n_\QQ} := \Big[ \sum_{i=0}^{n_{\QQ}} \enorm{v_i}^2 \Big]^{1/2}
	~
	\text{for all } \v \in \XX^{n_\QQ+1},
	\quad
	\enorm{\w}_{n_\CC} := \Big[ \sum_{j=1}^{n_{\CC}} \enorm{w_j}^2 \Big]^{1/2}
	~
	\text{for all } \w \in \XX^{n_\CC},
\end{equation*}
and note that $\XX_\coarse^{n_\QQ+1} \subseteq \XX^{n_\QQ+1}$ as well as $\XX_\coarse^{n_\CC} \subseteq \XX^{n_\CC}$ for all $\TT_\coarse \in \T$.
Then, the properties~\eqref{assumption:stab}--\eqref{assumption:drel} from Proposition~\ref{prop:axioms} also hold for $\widetilde{\eta}_\coarse$ and $\widetilde{\zeta}_\coarse$:

\begin{proposition}
\label{prop:axioms-sum}
	Let $\Cstab, \Crel, \Cdrel > 0$, and $0 < \qred < 1$ be the constants from Proposition~\ref{prop:axioms}.
	Then, for all $\TT_\coarse \in \T(\TT_0)$, and all $\TT_\fine \in \T(\TT_\coarse)$, there hold the following properties:
	
	\begin{enumerate}[label={(A\arabic*$^+$)}, ref={A\arabic*$^+$}]
		\bf
		\item Stability:
		\label{product:stab}
		\rm
		For all $\v_\fine \in \XX_\fine^{n_\QQ+1}$, $\v_\coarse \in \XX_\coarse^{n_\QQ+1}$, $\w_\fine \in \XX_\fine^{n_\CC}$, $\w_\coarse \in \XX_\coarse^{n_\CC}$, and $\UU_\coarse \subseteq \TT_\fine \cap \TT_\coarse$, it holds that
		\begin{align*}
			\big|
				\widetilde{\eta}_\fine(\UU_\coarse, \v_\fine)
				- \widetilde{\eta}_\coarse(\UU_\coarse, \v_\coarse)
			\big|
			&\leq
			\Cstab \, \enorm{\v_\fine - \v_\coarse}_{n_\QQ},\\
			\big|
				\widetilde{\zeta}_{\fine}(\UU_\coarse, \w_\fine)
				- \widetilde{\zeta}_{\coarse}(\UU_\coarse, \w_\coarse)
			\big|
			&\leq
			\Cstab \, \enorm{\w_\fine - \w_\coarse}_{n_\CC}.
		\end{align*}
		
		\bf
		\item Reduction:
		\label{product:red}
		\rm
		For all $\v_\coarse \in \XX_\coarse^{n_\QQ+1}$ and $\w_\coarse \in \XX_\coarse^{n_\CC}$, it holds that
		\begin{equation*}
			\widetilde{\eta}_{\fine}(\TT_\fine \backslash \TT_\coarse, \v_\coarse)
			\leq
			\qred \, \widetilde{\eta}_{\coarse}(\TT_\coarse \backslash \TT_\fine, \v_\coarse),
			\quad
			\widetilde{\zeta}_{\fine}(\TT_\fine \backslash \TT_\coarse, \w_\coarse)
			\leq
			\qred \, \widetilde{\zeta}_{\coarse}(\TT_\coarse \backslash \TT_\fine, \w_\coarse).
		\end{equation*}
		
		\bf
		\item Reliability:
		\label{product:rel}
		\rm
		The state and co-state components $\u_\coarse = (u_{\coarse,i})_{i=0}^{n_\QQ} \in \XX_\coarse^{n_\QQ+1}$ and $\z_\coarse = (z_{\coarse,j})_{j=1}^{n_\CC} \in \XX_\coarse^{n_\CC}$ satisfy that
		\begin{equation*}
			\enorm{\u - \u_{\coarse}}_{n_\QQ}
			\leq
			\Crel \, \widetilde{\eta}_{\coarse}
			\quad \text{and} \quad
			\enorm{\z - \z_{\coarse}}_{n_\CC}
			\leq
			\Crel \, \widetilde{\zeta}_{\coarse}.
		\end{equation*}
		
		\bf
		\item Discrete reliability:
		\label{product:drel}
		\rm
		The state and co-state components $\u_\coarse \in \XX_\coarse^{n_\QQ+1}$, $\u_\fine \in \XX_\fine^{n_\QQ+1}$, $\z_\coarse \in \XX_\coarse^{n_\CC}$, and $\z_\fine \in \XX_\fine^{n_\CC}$ satisfy that
		\begin{equation*}
			\enorm{\u_{\fine} - \u_{\coarse}}_{n_\QQ}
			\leq
			\Cdrel \, \widetilde{\eta}_{\coarse}(\TT_\coarse \backslash \TT_\fine),
			\quad
			\enorm{\z_{\fine} - \z_{\coarse}}_{n_\CC}
			\leq
			\Cdrel \, \widetilde{\zeta}_{\coarse}(\TT_\coarse \backslash \TT_\fine).
		\end{equation*}
	\end{enumerate}
\end{proposition}

\begin{proof}
	For stability~\eqref{product:stab} of the state, the inverse triangle inequality proves that
	\begin{align*}
		\big|
			\widetilde{\eta}_\fine(\UU_\coarse, \v_\fine)
			- \widetilde{\eta}_\coarse(\UU_\coarse, \v_\coarse)
		\big|
		&=
		\Big|
			\Big[ \sum_{i=0}^{n_{\QQ}} \eta_{\fine,i}(\UU_\coarse, v_{\fine,i})^2 \Big]^{1/2}
			- \Big[ \sum_{i=0}^{n_{\QQ}} \eta_{\coarse,i}(\UU_\coarse, v_{\coarse,i})^2 \Big]^{1/2}
		\Big|\\
		&\leq
		\Big|
			\sum_{i=0}^{n_{\QQ}} \big[ \eta_{\fine,i}(\UU_\coarse, v_{\fine,i}) 
			- \eta_{\coarse,i}(\UU_\coarse, v_{\coarse,i}) \big]^2
		\Big|^{1/2}\\
		&\eqreff{assumption:stab}{\leq}
		\Cstab \, \Big|
			\sum_{i=0}^{n_{\QQ}}  \enorm{v_{\fine,i} - v_{\coarse,i}}^2
		\Big|^{1/2}
		=
		\Cstab \, \enorm{\v_{\fine} - \v_{\coarse}}_{n_\QQ}.
	\end{align*}
	The estimate for the co-state follows analogously.
	Finally, the properties~\eqref{product:red}--\eqref{product:drel} follow directly from the corresponding properties from Proposition~\ref{prop:axioms}.
\end{proof}

Since the problems for the state and co-state components, \eqref{eq:primal-components} and~\eqref{eq:def-dual-components}, fit into the Lax--Milgram setting, there hold the Pythagoras identities
\begin{align*}
\label{eq:quasi-orth}
	\enorm{\u - \u_{\ell+n}}_{n_{\QQ}}^2
	+ \enorm{\u_{\ell+n} - \u_{\ell}}_{n_{\QQ}}^2
	&=
	\enorm{\u - \u_{\ell}}_{n_{\QQ}}^2,\\[1ex]
	\enorm{\z - \z_{\ell+n}}_{n_{\CC}}^2
	+ \enorm{\z_{\ell+n} - \z_{\ell}}_{n_{\CC}}^2
	&=
	\enorm{\z - \z_{\ell}}_{n_{\CC}}^2,
\end{align*}
for all $\ell, n \in \N_0$.
Overall, we get the following Proposition as an immediate consequence from, e.g., \cite[Theorem~12]{fpz16}.
From this, Theorem~\ref{th:linear-convergence} follows readily with~\eqref{eq:estimator-identity}.

\begin{proposition}
	Suppose \eqref{assumption:stab}--\eqref{assumption:rel} and $0 < \theta \leq 1$.
	Then, Algorithm~\ref{algorithm} guarantees linear convergence
	\begin{equation}
	\label{eq:cited-linear-convergence}
		\widetilde{\eta}_{\ell + n} \widetilde{\zeta}_{\ell + n}
		\leq
		\Clin \qlin^{n} \, \widetilde{\eta}_{\ell} \widetilde{\zeta}_{\ell}
		\quad \text{for all } \ell, n \in \N_0.
	\end{equation}
	The constants $\Clin > 0 $ and $0 < \qlin < 1$ depend only on $\Cstab$, $\qred$, $\Crel$, and the (arbitrary) adaptivity parameter $0 < \theta \le 1$.
	\qed
\end{proposition}

\subsection{Proof of optimal rates}
For all $r > 0$, we define the (combined) approximation classes
\begin{equation}
\label{eq:vector-approximation-classes}
\begin{split}
	\norm{\u}_{\mathbb{A}_r}
	&:=
	\sup_{N \in \N_0}
	(N+1)^r \min_{\TT_{\rm opt} \in \T(N)} \widetilde{\eta}_{\rm opt}(\u_{\rm opt})
	\in [0,\infty],\\
	\norm{\z}_{\mathbb{A}_r}
	&:=
	\sup_{N \in \N_0}
	(N+1)^r \min_{\TT_{\rm opt} \in \T(N)} \widetilde{\zeta}_{\rm opt}(\z_{\rm opt})
	\in [0,\infty].
\end{split}
\end{equation}
For these, we get the following result from, e.g., \cite[Theorem~13]{fpz16}.

\begin{proposition}
\label{prop:cited-optimal-rates}
	Let $0 < \theta < \theta_\opt := (1+\Cstab^2\Cdrel^2)^{-1}$.
	Suppose that the set of marked elements $\MM_\ell$ in Algorithm~\ref{algorithm}{\rm(iv)} has minimal cardinality.
	Let $s, t > 0$ with $\norm{\u}_{\mathbb{A}_{s}} + \norm{\z}_{\mathbb{A}_{t}} < \infty$.
	Then, there exists a constant $\expandafter\widetilde\Copt>0$ such that
	\begin{equation}
	\label{eq:cited-optimal-rates}
		\widetilde{\eta}_{\ell} \widetilde{\zeta}_{\ell}
		\leq
		\expandafter\widetilde\Copt \,
		\norm{\u}_{\mathbb{A}_{s}} \norm{\z}_{\mathbb{A}_{t}}
		\big( \#\TT_\ell - \#\TT_0 \big)^{-(s+t)}
		\quad
		\text{for all } \ell \in \N_0.
	\end{equation}
	The constant $\expandafter\widetilde\Copt$ depends only on $\Cstab$, $\qred$, $\Crel$, $\Cdrel$, $\Cmark$, $\theta$, $s$, $t$, and the properties~\eqref{eq:mesh-sons}--\eqref{eq:mesh-closure} of the mesh-refinement.
	\qed
\end{proposition}

Finally, Theorem~\ref{th:optimal-rates} follows from Proposition~\ref{prop:cited-optimal-rates} by relating the different approximation classes used in both results.

\begin{proof}[Proof of Theorem~\ref{th:optimal-rates}]
	Let $r > 0$.
	We show that $\sum_{i=0}^{n_\QQ} \norm{u_i}_{\A_r}^2 \simeq \norm{\u}_{\A_r}^2$.
	From the definitions~\eqref{eq:approximation-classes} and~\eqref{eq:vector-approximation-classes} of the approximation classes, we immediately see that, for all $i = 0, \ldots, n_\QQ$,
	\begin{align*}
		\norm{u_i}_{\A_r}
		&=
		\sup_{N \in \N_0} (N+1)^r \min_{\TT_{\rm opt} \in \T(N)} \eta_{{\rm opt},i}(u_{{\rm opt},i})\\
		&\leq
		\sup_{N \in \N_0} (N+1)^r \min_{\TT_{\rm opt} \in \T(N)} \Big[ \sum_{i=0}^{n_\QQ} \eta_{{\rm opt},i}(u_{{\rm opt},i})^2 \Big]^{1/2}
		=
		\norm{\u}_{\A_r}.
	\end{align*}
	Summing the last estimate for all $i = 0, \ldots, n_\QQ$, we obtain that
	\begin{equation}
	\label{eq:proof-equiv-1}
		\frac{1}{n_\QQ + 1} \sum_{i=0}^{n_\QQ} \norm{u_i}_{\A_r}^2
		\leq
		\norm{\u}_{\A_r}^2.
	\end{equation}
	
	For the converse estimate, we fix $N \in \N$ and define $K := \lfloor N/(n_\QQ + 1) \rfloor$.
	Let further be $\TT_k \in \TT(K)$ for $k = 0, \ldots, n_\QQ$ such that
	\begin{equation*}
		\eta_{k,k}(u_{k,k})
		=
		\min_{\TT_{\rm opt} \in \T(K)} \eta_{{\rm opt},k}(u_{{\rm opt},k}).
	\end{equation*}
	With the overlay estimate~\eqref{eq:mesh-overlay}, we have that
	\begin{equation*}
		\#\bigoplus_{k=0}^{n_\QQ} \TT_k
		=
		\Big[ \sum_{k=0}^{n_\QQ} \#\TT_k \Big] - n_\QQ \#\TT_0
		=
		\Big[ \sum_{k=0}^{n_\QQ} (\#\TT_k - \#\TT_0) \Big] + \#\TT_0
		=
		\Big[ \sum_{k=0}^{n_\QQ} K \Big] + \#\TT_0
		=
		N + \#\TT_0.
	\end{equation*}
	Therefore, it holds that $ \TT_\triangle := \bigoplus_{k=0}^{n_\QQ} \TT_k \in \T(N)$.
	From this, we infer that
	\begin{equation*}
		\min_{\TT_{\rm opt} \in \T(N)} \Big[ \sum_{i=0}^{n_\QQ} \eta_{{\rm opt},i}(u_{{\rm opt},i})^2 \Big]^{1/2}
		\leq
		\Big[ \sum_{i=0}^{n_\QQ} \eta_{\triangle,i}(u_{\triangle,i})^2 \Big]^{1/2}
		\, \eqreff{eq:linear-convergence}{\leq}
		\Clin \, \Big[ \sum_{i=0}^{n_\QQ} \eta_{i,i}(u_{i,i})^2 \Big]^{1/2}.
	\end{equation*}
	Multiplying this by $(N+1)^r$, we obtain that
	\begin{align*}
	(N+1)^r &\min_{\TT_{\rm opt} \in \T(N)} \Big[ \sum_{i=0}^{n_\QQ} \eta_{{\rm opt},i}(u_{{\rm opt},i})^2 \Big]^{1/2}
	\leq
	\Clin \, (N+1)^r \Big[ \sum_{i=0}^{n_\QQ} \eta_{i,i}(u_{i,i})^2 \Big]^{1/2}\\
	&=
	\Clin \, \Big(\frac{N+1}{K+1}\Big)^r (K+1)^r \Big[ \sum_{i=0}^{n_\QQ} \eta_{i,i}(u_{i,i})^2 \Big]^{1/2}\\
	&=
	\Clin \, \Big(\frac{N+1}{K+1}\Big)^r (K+1)^r \Big[ \sum_{i=0}^{n_\QQ} \min_{\TT_{\rm opt} \in \T(K)} \eta_{{\rm opt},i}(u_{{\rm opt},i})^2 \Big]^{1/2}\\
	&\leq
	\Clin \, \Big(\frac{N+1}{K+1}\Big)^r \Big[ \sum_{i=0}^{n_\QQ} \norm{u_i}_{\A_r}^2 \Big]^{1/2}.
	\end{align*}
	Taking the supremum over all $N \in \N$ of the last estimate and using $(N+1)/(K+1) \leq n_\QQ + 2$, we finally arrive at
	\begin{equation}
	\label{eq:proof-equiv-2}
		\norm{\u}_{\A_r}
		\leq
		\Clin \, (n_\QQ + 2)^{r} \Big[ \sum_{i=0}^{n_\QQ} \norm{u_i}_{\A_r}^2 \Big]^{1/2}.
	\end{equation}
	Thus, by combining~\eqref{eq:proof-equiv-1}--\eqref{eq:proof-equiv-2}, we have that
	\begin{equation*}
		\frac{1}{n_\QQ + 1} \sum_{i=0}^{n_\QQ} \norm{u_i}_{\A_s}^2
		\leq
		\norm{\u}_{\A_s}^2
		\leq
		\Clin^2 (n_\QQ + 2)^{2s} \sum_{i=0}^{n_\QQ} \norm{u_i}_{\A_s}^2.
	\end{equation*}
	Clearly, it thus holds $\norm{\u}_{\A_s} < \infty$ if and only if $s = \min \set{s_i}{0 \leq i \leq n_{\QQ}}$.
	Analogously, it follows that
	\begin{equation*}
		\frac{1}{n_\CC} \sum_{j=1}^{n_\CC} \norm{z_j}_{\A_t}^2
		\leq
		\norm{\z}_{\A_t}^2
		\leq
		\Clin^2 \, (n_\CC+1)^{2t} \sum_{j=1}^{n_\CC} \norm{z_j}_{\A_t}^2,
	\end{equation*}
	and $\norm{\z}_{\A_t} < \infty$ if and only if $t = \min \set{t_j}{1 \leq j \leq n_{\CC}}$.
	
	Finally, combining the last statements with the statement of Proposition~\ref{prop:cited-optimal-rates}, we conclude the proof with $\Copt:= \Clin^2 (n_\QQ + 2)^{s+1/2} (n_\CC+1)^{t+1/2} \expandafter\widetilde\Copt$.
\end{proof}


\section{Numerical examples}\label{sec:numerics}

For the following examples, we consider the initial mesh $\TT_0$ of $\Omega := (0,1)^2 \subseteq \R^2$ shown in Figure~\ref{fig:initial-mesh} with the sets
\begin{align*}
	T_1 &:= \set{x \in \R^2}{x_1 + x_2 > 3/2} \cap \Omega,\\
	T_2 &:= \set{x \in \R^2}{x_1 + x_2 < 1/2} \cap \Omega,\\
	T_3 &:= \set{x \in \R^2}{\max\{ x_1, x_2 \} < 1/4} \cap \Omega.
\end{align*}

With the characteristic function $\chi_\omega$ of a measurable subset $\omega \subseteq \Omega$, we define
\begin{align*}
	f_1 &:= 4 x_1(1-x_1) + 4 x_2(1-x_2),
	\quad
	f_2 := 5 \pi^2 \sin(\pi x_1) \sin(2\pi x_2),\\
	\g_1 &:= (1,0)^\intercal \chi_{T_1},
	\quad
	\g_2 := (-1,0)^\intercal \chi_{T_2},
	\quad
	g_3 := \chi_{T_3},
\end{align*}
as well as $f_0 = g_1 = g_2 = 0$ and $\f_0 = \f_1 = \f_2 = \g_3 = \vec{0}$.
In all our experiments, we set $\alpha = 0$.
Since all $f_i, g_j, \f_i, \g_j$ are linearly independent, the matrix $\B^\intercal \B$ has full rank.
In particular, condition~\eqref{eq:locally-convex} is satisfied; see Remark~\ref{rem:ls-solve}.
As marking parameter, we use $\theta = 0.5$.

\begin{figure}
	\centering
		
\begin{tikzpicture}[x=0.3\linewidth,y=0.3\linewidth]
\coordinate (P0) at (0.00,0.00);
\coordinate (P1) at (1.00,0.00);
\coordinate (P2) at (1.00,1.00);
\coordinate (P3) at (0.00,1.00);
\coordinate (P4) at (0.00,0.50);
\coordinate (P5) at (1.00,0.50);
\coordinate (P6) at (0.50,1.00);
\coordinate (P7) at (0.50,0.00);

\fill[fill=lightgray] (P0) -- (P7) -- (P4) -- cycle;
\fill[fill=lightgray] (P2) -- (P5) -- (P6) -- cycle;
\fill[fill=lightgray] (0.25,0.25) -- (0.75,0.25) -- (0.75,0.75) -- (0.25,0.75) -- cycle;

\draw (P0) -- (P1) -- (P2) -- (P3) -- cycle;
\draw (P0) -- (P2);
\draw (P1) -- (P3);
\draw (P4) -- (P5);
\draw (P6) -- (P7);
\draw (P4) -- (P6);
\draw (P5) -- (P7);
\draw (P4) -- (P7);
\draw (P5) -- (P6);
\draw (0.25,0) -- (0.25,1);
\draw (0.75,0) -- (0.75,1);
\draw (0,0.25) -- (1,0.25);
\draw (0,0.75) -- (1,0.75);

\node[at={(0.93,0.82)}] {\Large $T_1$};
\node[at={(0.18,0.07)}] {\Large $T_2$};
\node[at={(0.59,0.32)}] {\Large $T_3$};

\end{tikzpicture}%
	\hspace{1em}
	\includegraphics[width=0.3\linewidth]{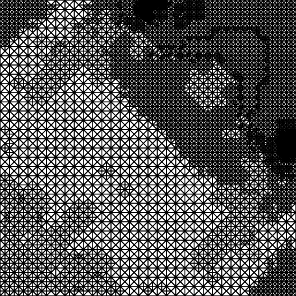}
	\hspace{1em}
	\includegraphics[width=0.3\linewidth]{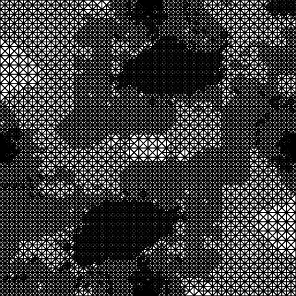}
	\caption{Left: Initial mesh $\TT_0$ of the unit square $(0,1)^2$ for numerical experiments.
	Middle: Mesh with $\#\TT_{10} = 11096$ elements for the setting of Section~\ref{subsec:numerics-1}.
	Right: Mesh with $\#\TT_9 = 16080$ elements for the setting of Section~\ref{subsec:numerics-2}.}
	\label{fig:initial-mesh}
\end{figure}

\subsection{Single parameter and measurement}\label{subsec:numerics-1}

For our first experiment, we consider the following parametrized discrete PDE problem with parameter $\p \in \QQ:= \R$:
Find $u_\coarse(\p) \in \XX_\coarse$ such that
\begin{equation}
\label{eq:numerics-problem-1}
	a(u_\coarse(\p), v_\coarse)
	:=
	\int_\Omega	\nabla u_\coarse(\p) \cdot \nabla v_\coarse \d{x}
	=
	p_1 \int_\Omega f_1 v_\coarse \d{x}
	=:
	b(\p, v)
	\quad
	\text{for all } v_\coarse \in \XX_\coarse.
\end{equation}
The exact (continuous) solution of this problem is known to be
\begin{equation*}
	u(\p)
	=
	p_1 u_1
	=
	p_1 x_1 x_2 (1-x_1) (1-x_2).
\end{equation*}
We further suppose that we have one measurement (corresponding to $\p^\star = 1$)
\begin{equation*}
	\G^\star
	=
	\frac{11}{960}
	=
	\G(u(\p^\star))
	:=
	- \int_\Omega \g_1 \cdot \nabla u(\p^\star) \d{x}
	=
	-\int_{T_1} \frac{\partial u(\p^\star)}{\partial x_1} \d{x}.
\end{equation*}
We compute an approximation to $\p^\star$ by two methods:
\begin{itemize}
	\item The approximations $\p_\ell^\star$ are obtained by our adaptive Algorithm~\ref{algorithm}, which is driven by the estimator $\varrho_{\ell}$ from~\eqref{eq:weighted-estimator}.
	
	\item The approximations $\overline{\p}_\ell^\star$ are obtained by our adaptive algorithm, where $\varrho_\ell$ is substituted by $\overline{\varrho}_\ell(\overline{\p}_\ell^\star) := \eta_\ell(\overline{\p}_\ell^\star) + \zeta_\ell(\overline{\p}_\ell^\star)$, which is the prevalent error estimator from the existing literature on AFEM for optimal control problems~\cite{bm11,gy17}.
	Here, for $\p \in \QQ$, $\eta_\coarse(\p)$ and $\zeta_\coarse(\p)$ are the residual error estimators of the energy errors of $u_\coarse(\p)$ from~\eqref{eq:problem} and $z_\coarse(\p)$ from~\eqref{eq:z-equation}, respectively.
\end{itemize}
The results can be seen in Figure~\ref{fig:singleparameter}.
We see that the classical estimator $\overline{\varrho}_\ell(\overline{\p}_\ell^\star)$ drastically underestimates the rate of the parameter error, whereas our estimator $\varrho_\ell$ matches it perfectly.
Also, the parameter error of our approach is uniformly better by some (small) multiplicative factor.
However, this effect is negligible for large $\#\TT_\ell$.

\begin{figure}
	\centering
		
\pgfplotstableread[col sep = comma]{figures/singleParameterSumEstimator.dat}{\sumEst}%
\pgfplotstableread[col sep = comma]{figures/singleParameterProductEstimator.dat}{\prodEst}%

\begin{tikzpicture}[scale=1]
\begin{loglogaxis}[xlabel=$\#\mathcal{T}_\ell$, ylabel={error / estimator},
width=0.6\textwidth,
legend pos=south west, legend columns=1, 
legend style={
	draw=none,
	fill=none,
	font=\small,
}, legend cell align={left}]

	\addplot [red,thick,mark=o] table [x={nElements}, y={errEstimator}] {\prodEst};
	\addplot [red,thick,mark=x] table [x={nElements}, y={paramError}] {\prodEst};
	\addplot [blue,thick,mark=o] table [x={nElements}, y={errEstimator}] {\sumEst};
	\addplot [blue,thick,mark=x] table [x={nElements}, y={paramError}] {\sumEst};
	
	\addplot[dashed] table [x={nElements}, y expr={2/\thisrowno{0}}] {\prodEst} node[below,rotate=-37] at (axis cs:5E4,5E-5) {\tiny{$\propto (\#\mathcal{T}_\ell)^{-1}$}};
	\addplot[dashed] table [x={nElements}, y expr={3/sqrt(\thisrowno{0})}] {\prodEst} node[below,rotate=-18] at (axis cs:5E4,8E-2) {\tiny{$\propto (\#\mathcal{T}_\ell)^{-1/2}$}};
	
	\legend{$\varrho_\ell$ from \eqref{eq:weighted-estimator},%
		$\| \boldsymbol{p}^\star - \boldsymbol{p}^\star_\ell \|_{\mathcal{Q}}$,%
		$\overline{\varrho}_\ell(\overline{\boldsymbol{p}}^\star_\ell) = \eta_\ell(\overline{\boldsymbol{p}}^\star_\ell) + \zeta_\ell(\overline{\boldsymbol{p}}^\star_\ell)$,%
		$\| \boldsymbol{p}^\star - \overline{\boldsymbol{p}}^\star_\ell \|_{\mathcal{Q}}$}
\end{loglogaxis}

\end{tikzpicture}
	\caption{Results for the problem from Section~\ref{subsec:numerics-1}.
	The $\p^\star_\ell$ are computed by our adaptive Algorithm~\ref{algorithm} driven by the estimator $\varrho_\ell$; the $\overline{\p}_\ell^\star$ are computed by an algorithm driven by $\overline{\varrho}_\ell(\overline{\p}_\ell^\star) = \eta_\ell(\overline{\p}_\ell^\star) + \zeta_\ell(\overline{\p}_\ell^\star)$.}
	\label{fig:singleparameter}
\end{figure}

\subsection{Multiple parameters and measurements with perturbation}\label{subsec:numerics-2}

For our second experiment, we consider the following parametrized discrete PDE problem with parameter $\p \in \QQ:= \R^2$:
Find $u_\coarse(\p) \in \XX_\coarse$ such that
\begin{equation}
\label{eq:numerics-problem-2}
\begin{split}
	a(u_\coarse(\p), v_\coarse)
	&:=
	\int_\Omega	\nabla u_\coarse(\p) \cdot \nabla v_\coarse \d{x}\\
	&=
	p_1 \int_\Omega f_1 v_\coarse \d{x} + p_2 \int_\Omega f_2 v_\coarse \d{x}
	=:
	b(\p, v)
	\quad
	\text{for all } v_\coarse \in \XX_\coarse.
\end{split}
\end{equation}

The exact (continuous) solution of this problem is known to be
\begin{equation*}
	u(\p)
	=
	p_1 u_1 + p_2 u_2
	=
	p_1 x_1 x_2 (1-x_1) (1-x_2) + p_2 \sin(\pi x_1) \sin(2\pi x_2).
\end{equation*}
We further suppose that we have three \emph{exact} measurements (corresponding to the exact parameter $\p^\star = (2, 1/2)^\intercal$)
\begin{align*}
	\overline{\G}^\star
	&=
	\Big( \frac{11\pi+160}{480\pi}, \frac{11\pi-160}{480\pi}, \frac{121}{4608}  \Big)^\intercal
	=
	\big( G_1(u(\p^\star)), G_2(u(\p^\star)), G_3(u(\p^\star)) \big)^\intercal\\
	&:=
	\Big( - \int_\Omega \g_1 \cdot \nabla u(\p^\star) \d{x},
	- \int_\Omega \g_2 \cdot \nabla u(\p^\star) \d{x},
	\int_\Omega g_3 u(\p^\star) \d{x} \Big)^\intercal\\
	&=
	\Big( -\int_{T_1} \frac{\partial u(\p^\star)}{\partial x_1} \d{x},
	\int_{T_2} \frac{\partial u(\p^\star)}{\partial x_1} \d{x},
	\int_{T_3} u(\p^\star) \d{x} \Big)^\intercal.
\end{align*}
These exact measurements are perturbed by Gaussian random noise $X_k \sim N(0,\sigma^2)$ for $i = 1,2,3$ and for some standard deviation $\sigma \geq 0$, such that $\G^\star = \overline{\G}^\star + (X_1, X_2, X_3)^\intercal$.
Note that, for non-vanishing perturbation, $\G^\star$ does not necessarily coincide with $\G(u(\p^\star))$ anymore.
Likewise, the sequence $(\p_\ell^\star)_{\ell \in \N}$ does not converge to the given parameter $\p^\star$ but rather to some $\overline{\p}^\star \in \QQ$ which is the least-squares solution to~\eqref{eq:ls-functional} on the continuous level with the perturbed measurements $\G^\star$.

We compute $\p^\star_\ell$ by Algorithm~\ref{algorithm} with different levels of perturbation.
The results can be seen in Figure~\ref{fig:multipleparameters}.
We see that, for the different levels of perturbation, the parameter error cannot fall beyond a threshold that depends on the standard deviation $\sigma$ of the perturbation.
This is to be expected, since inference of parameters from experiments is limited by the accuracy of the measurement.
Our estimator, however, is independent of the discrete parameter and measurements and, hence, continues to converge independently of the perturbation.
In particular, the matrix $\B_\ell$ on the finest level of the algorithm can be stored and reused to compute a refined parameter value if a new set of measurements with improved accuracy becomes available.
This is not the case for the sum estimator $\overline{\varrho}_\ell(\overline{\p}_\ell^\star)$ from the last section, since it explicitly depends on the parameter estimates.

\begin{figure}
	\centering
\pgfplotstableread[col sep = comma]{figures/multipleParameters.dat}{\est}%

\begin{tikzpicture}[scale=1]
\begin{loglogaxis}[xlabel=$\#\mathcal{T}_\ell$, ylabel={error / estimator},
ymin=5e-10, ymax=1E1,
width=0.6\textwidth,
legend pos=south west, legend columns=1, 
legend style={
	draw=none,
	fill=none,
	font=\small,
}, legend cell align={left}]

	\addplot [blue,thick,mark=o] table [x={nElements}, y={errEstimator}] {\est};
	\addplot [red,thick,mark=triangle] table [x={nElements}, y={paramError_1}] {\est};
	\addplot [red,thick,mark=square] table [x={nElements}, y={paramError_2}] {\est};
	\addplot [red,thick,mark=pentagon] table [x={nElements}, y={paramError_3}] {\est};
	\addplot [red,thick,mark=o] table [x={nElements}, y={paramError_4}] {\est};
	
	\addplot[dashed] table [x={nElements}, y expr={2/\thisrowno{0}}] {\est} node[below,rotate=-27] at (axis cs:1E4,2E-4) {\tiny{$\propto (\#\mathcal{T}_\ell)^{-1}$}};
	
	\legend{$\varrho_\ell$ from \eqref{eq:weighted-estimator},%
		{$\| \boldsymbol{p}^\star - \boldsymbol{p}^\star_\ell \|_{\mathcal{Q}}, ~\sigma = 10^{-3}$},%
		{$\| \boldsymbol{p}^\star - \boldsymbol{p}^\star_\ell \|_{\mathcal{Q}}, ~\sigma = 10^{-5}$},%
		{$\| \boldsymbol{p}^\star - \boldsymbol{p}^\star_\ell \|_{\mathcal{Q}}, ~\sigma = 10^{-7}$},%
		{$\| \boldsymbol{p}^\star - \boldsymbol{p}^\star_\ell \|_{\mathcal{Q}}, ~\sigma = 0$}}
\end{loglogaxis}

\end{tikzpicture}
	\caption{Results for the problem from Section~\ref{subsec:numerics-2}.
	Differently marked lines are obtained by perturbing the true measurements by Gaussian noise with standard deviation $\sigma$.}
	\label{fig:multipleparameters}
\end{figure}
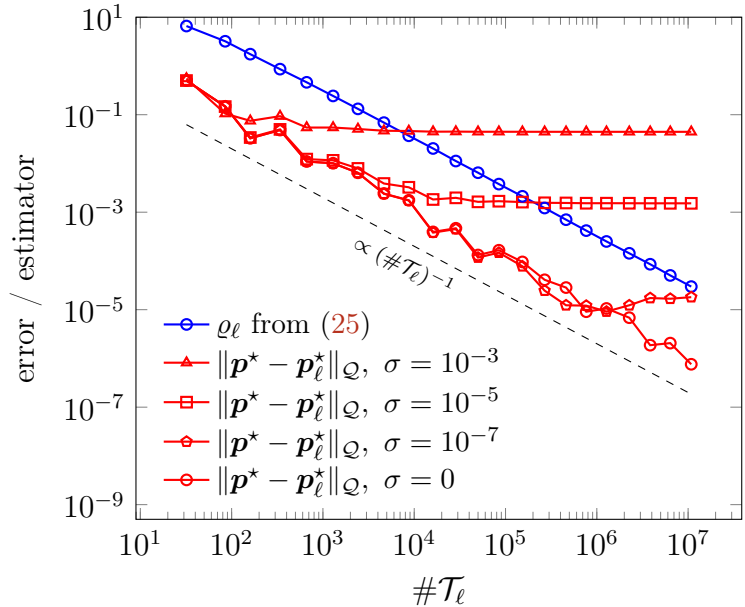



{
	\renewcommand{\section}[3][]{\vskip4mm\begin{center}\bf\normalsize R\small EFERENCES\normalsize\end{center}\vskip2mm}
	
	\bibliographystyle{alpha}
	\bibliography{literature}
}

\end{document}